\documentclass[12pt]{amsart}

\usepackage[all]{xy}
\usepackage{fullpage}
\usepackage{latexsym}
\usepackage{amsmath}
\usepackage{amsfonts}
\usepackage{amssymb}
\usepackage{amsthm}
\usepackage{eucal}
\usepackage{enumerate,yfonts}
\usepackage{mathrsfs}
\usepackage{graphicx}
\usepackage{graphics}
\usepackage{epstopdf}
\usepackage{amscd}
\usepackage{bbm}
\usepackage{hyperref}
\usepackage{url}
\usepackage{color}
\usepackage{cancel}
\usepackage{enumerate}
\usepackage{amsmath,amsthm}
\usepackage{amssymb}
\usepackage{epsfig}
\usepackage{pstricks}
\usepackage{xy}
\usepackage{xypic}

\newtheorem{thm}{Theorem}[section]
\newtheorem{corollary}[thm]{Corollary}
\newtheorem{lemma}[thm]{Lemma}

\newtheorem{proposition}[thm]{Proposition}
\newtheorem{prop}[thm]{Proposition}

\newtheorem{thm-dfn}[thm]{Theorem-Definition}

\newtheorem{claim}[thm]{Claim}
\newtheorem{constructionlemma}[thm]{Construction-Lemma}

\theoremstyle{definition}
\newtheorem{definition}[thm]{Definition}
\newtheorem{remark}[thm]{Remark}
\newtheorem{example}[thm]{Example}

\numberwithin{equation}{section}

\newcommand{\fg}{{\mathfrak g}}
\newcommand{\ft}{{\mathfrak t}}
\newcommand{\fl}{{\mathfrak l}}
\newcommand{\fb}{{\mathfrak b}}
\newcommand{\fu}{{\mathfrak u}}
\newcommand{\fp}{{\mathfrak p}}

\newcommand{\fa}{{\mathfrak a}}

\newcommand{\fk}{{\mathfrak k}}
\newcommand{\fm}{{\mathfrak m}}
\newcommand{\fn}{{\mathfrak n}}

\newcommand{\rW}{{\mathrm W}}

\newcommand{\bA}{{\mathbb A}}
\newcommand{\bC}{{\mathbb C}}
\newcommand{\bX}{{\mathbb X}}
\newcommand{\bG}{{\mathbb G}}

\newcommand{\bD}{{\mathbb D}}
\newcommand{\bP}{{\mathbb P}}

\newcommand{\mY}{\mathcal{Y}}
\newcommand{\mX}{\mathcal{X}}
\newcommand{\mE}{\mathcal{E}}
\newcommand{\mF}{\mathcal{F}}

\newcommand{\mA}{\mathcal{A}}
\newcommand{\mM}{\mathcal{M}}

\newcommand{\mO}{\mathcal{O}}

\newcommand{\mP}{\mathcal{P}}

\newcommand{\mG}{\mathcal{G}}
\newcommand{\mN}{\mathcal{N}}
\newcommand{\mB}{\mathcal{B}}

\newcommand{\on}{\operatorname}

\newcommand{\tX}{\widetilde X}
\newcommand{\tY}{\widetilde Y}

\newcommand{\tG}{\widetilde G}
\newcommand{\tK}{\widetilde K}

\newcommand{\ra}{\rightarrow}
\newcommand{\la}{\leftarrow}

\newcommand{\bs}{\backslash}
\newcommand{\is}{\simeq}

\newcommand{\Av}{\on{Av}}
\newcommand{\av}{\on{av}}

\newcommand{\Loc}{\on{LocSys}}

\newcommand{\quash}[1]{}  
\newcommand{\nc}{\newcommand}

\newcommand{\frakg}{{\mathfrak g}}

\newcommand{\frakn}{{\mathfrak n}}

\newcommand{\fraks}{{\mathfrak s}}

\newcommand{\bbA}{{\mathbb A}}

\newcommand{\bbC}{{\mathbb C}}
\newcommand{\bbD}{{\mathbb D}}

\newcommand{\bbG}{{\mathbb G}}

\newcommand{\calF}{{\mathcal F}}
\newcommand{\calG}{{\mathcal G}}
\newcommand{\calH}{{\mathcal H}}

\newcommand{\calM}{{\mathcal M}}

\newcommand{\calO}{{\mathcal O}}

\newcommand{\beq}{\begin{equation}}
\newcommand{\eeq}{\end{equation}}

\newcommand{\oblv}{\on{oblv}}

\nc{\al}{{\alpha}} \nc{\be}{{\beta}} \nc{\ga}{{\gamma}}
\nc{\ve}{{\varepsilon}} \nc{\Ga}{{\Gamma}} 
\nc{\La}{{\Lambda}}

\nc{\ad }{{\on{ad }}}

\nc{\aff}{{\on{aff}}} \nc{\Aff}{{\mathbf{Aff}}}

\nc{\der}{{\on{der}}}

\nc{\diag}{{\on{diag}}}

\nc{\Fl}{{\calF\ell}}

\nc{\Hg}{{\on{Higgs}}}
\newcommand{\Hom}{{\on{Hom}}}

\nc{\Id}{{\on{Id}}}
\newcommand{\ind}{{\on{ind}}}
\nc{\Ind}{{\on{Ind}}}
\newcommand{\Lie}{{\on{Lie}}}
\nc{\Op}{{\on{Op}}}

\nc{\res}{{\on{res}}}

\nc{\tr}{{\on{tr}}}

\nc{\GSp}{{\on{GSp}}} \nc{\GU}{{\on{GU}}} \nc{\SL}{{\on{SL}}}
\nc{\SU}{{\on{SU}}} \nc{\SO}{{\on{SO}}}

\nc{\nh}{{\Loc_{J^p}(\tau')}}
\nc{\bnh}{{\Loc_{\breve J^p}(\tau')}}

\nc{\bU}{{\overline{U}}} \nc{\IC}{{\on{IC}}}

\newcommand{\UHom}{\underline{\Hom}}

\newcommand{\mtB}{\widetilde{\mB}}
\newcommand{\wt}{\widetilde}

\newcommand{\beqn}{\begin{equation*}}
\newcommand{\eeqn}{\end{equation*}}

\newcommand{\niceG}{\mathcal G}
\newcommand{\niceF}{\mathcal F}

\begin{document}
\title{A formula for the geometric Jacquet functor and its character sheaf analogue}
        \author{Tsao-Hsien Chen and Alexander Yom Din}
        \address{Department of Mathematics, University of Chicago, Chicago, IL 60637, USA.}
        \email{chenth@math.uchicago.edu}
        \address{Department of Mathematics, California Institute of Technology, Pasadena, CA 91125, USA.}
         \email{ayomdin@caltech.edu}

\thanks{}
\thanks{}

\maketitle
\begin{abstract}

Let $(G,K)$ be a symmetric pair over the complex numbers, and let $X=K \backslash G$ be the corresponding symmetric space. In this paper we study a nearby cycles functor associated to a degeneration of $X$ to $MN \backslash G$, which we call the ``wonderful degeneration". We show that on the category of character sheaves on $X$, this functor is isomorphic to a composition of two averaging functors (a parallel result, on the level of functions in the $p$-adic setting, was obtained in \cite{BK, SV}). As an application, we obtain a formula for the geometric Jacquet functor of \cite{ENV} and use this formula to give a geometric proof of the celebrated Casselman's submodule theorem and establish a second adjointness theorem for Harish-Chandra modules.
\end{abstract}

\setcounter{tocdepth}{1} \tableofcontents


\section{Introduction}
\subsection{Some notation}
Let $G$ be a (connected) reductive algebraic group. Let $\theta:G\ra G$ be an involution, and $K$ an open subgroup of $G^{\theta} = \{g\in G \ | \ \theta(g)=g\}$. We call the pair $(G,K)$ a symmetric pair and write $X := K\backslash G$ for the associated symmetric space.
Let $P$ be a minimal $\theta$-split parabolic subgroup of $G$, and $P=MAN$ its Langlands decomposition (see \S\ref{symmetric spaces} for the precise definitions).

For a morphism $g : \tY \to \bP^1$, write $Y_t := g^{-1}(t)$ and $Y^{\circ} := g^{-1} (\bG_m)$. Denote by $D(\cdot)$ the bounded derived category of constructible sheaves/holonomic $D$-modules/$\ell$-adic sheaves on $\cdot$ (see \S\ref{Sec:NotationGeneral}).

\subsection{Main result}

In this paper we consider a ``degeneration" of $X$ to the base affine space $MN\backslash G$; That is, a smooth variety $\tX$, equipped with a $\bG_m$-action, and a smooth map \[f:\tX\ra\bP^1\] intertwining the $\bG_m$-action on $\tX$ with the standard one on $\bP^1$, such that $X_1 \is X$, $X_0 \is MN\backslash G$ and $X_{\infty} \is M\bar N \backslash G$. Note that we have a canonical $\bG_m$-equivariant identification $X^{\circ} \cong \bG_m \times X_1$. We call this degeneration the \textbf{wonderful degeneration} of $X$.

We explain briefly the construction of $\tX$ (for details see \S\ref{Sec:LimitSchemes}).
Let $\gamma:\bG_m\ra A$ be a co-character of $A$ which is negative on roots of $N$. One can associate to it a
subgroup scheme $\tK\ra \bP^1$ of the constant group scheme $\tG:=\bP^1\times G \to \bP^1$, such that $K_{t} = \gamma(t) K \gamma(t)^{-1}$ for $t \in \bG_m$, $K_{0} = MN$ and $K_{\infty} = M \bar N$. Then the wonderful degeneration of the symmetric space is defined as the quotient $f: \tX := \tK \backslash \tG \ra \bP^1$.

The main object of this paper is the following \textbf{limit functor} associated to the wonderful degeneration:
\beq
L_{\gamma} := \psi_f \circ a_f^*:D(X_1) \ra D(X_0).
\eeq
Here $a_f:X^{\circ} \cong \bG_m \times X_1 \ra X_1$ is the projection map and $\psi_f$ is the nearby cycles functor associated to the map $f:\tX \ra \bP^1$. The limit functor $L_{\gamma}$ is $t$-exact w.r.t. the perverse $t$-structure.

To state our main result, we need to recall the notion of averaging functors (see \S\ref{averaging functor}).
By \cite{BL}, we have the following $!$-averaging and $*$-averaging functors:

 \[\on{Av}_{M\bar N}^{K}:D(X_1)\is D(K \bs G) \is D_K(G)\ra D_{M\bar N}(G) \is D(M \bar N \bs G) \is D(X_{\infty})\] and
\[\on{av}_{MN}^{M\bar N}:D(X_{\infty}) \is D(M \bar N \bs G) \is D_{M\bar N}(G)\ra D_{MN}(G) \is D(MN \bs G) \is D(X_0).\]
Recall that $\on{Av}_{M\bar N}^{K}$ (resp. $\on{av}_{MN}^{M\bar N}$) is, roughly, the composition of
$*$-pull back with
$!$-push forward (resp. $*$-push forward) along the correspondence $K\backslash G\la M \backslash G\ra M\bar N\backslash G$ (resp. $M\bar N\backslash G\la M \backslash G\ra MN\backslash G$).
For any algebraic subgroup $H\subset G$, 
the limit functor 
and the averaging functors above have natural upgrades to functors 
between equivariant categories 
$D(X_1 / H) := D_{K\times H}(G)$, $D(X_0 / H):=
D_{MN\times H}(G)$, $D (X_\infty / H ):=D_{M\bar N\times H}(G)$
and we use the same notation $L_\gamma$, $\on{Av}_{MN}^{K}$,
$\on{Av}_{M\bar N}^{K}$
for the corresponding functors.

Let $D\mathcal{CS}(X_1 / K)\subset D(X_1 / K)=D_{K\times K}(G)$ be the full subcategory of character
sheaves on $X_1$ introduced by Lusztig, Grojnowski, and Ginzburg (see \S\ref{Sec:AdmShv}).
Here is the main result of this paper:

\begin{thm}\label{main thm} There is a natural isomorphism of functors
$D\mathcal{CS}(X_1 / K)\ra D(X_0 / K)$
\[L_{\gamma} \is\on{av}_{MN}^{M\bar N}\circ\on{Av}_{M\bar N}^{K}.\]
\end{thm}

Two corollaries of theorem \ref{main thm} are the following:

\begin{corollary}\label{exactness for av}
The functor $\on{av}_{MN}^{M\bar N}\circ\on{Av}_{M\bar N}^{K} : D\mathcal{CS}(X_1 / K) \to D(X_0 / K)$ 
is $t$-exact (w.r.t. the perverse $t$-structure).
\end{corollary}

\begin{corollary}\label{two adjoints}
The functor $L_{\gamma}$ admits both a left and a right adjoint.
\end{corollary}

In the group case, that is for $G= H \times H$  where $H$ is a reductive group and $K= \Delta H \subset G$ is the diagonal copy of $H$, theorem
\ref{main thm} and its corollary \ref{exactness for av} are ones of the main results in \cite{BFO}, which is a key step in their work on classification of character D-modules.
We would like to emphasize that our argument here is different from the one in \cite{BFO}; it is geometric and can be applicable also to the $\ell$-adic and D-module settings.

\remark{
The formula in theorem \ref{main thm} resembles formulas in \cite{BK,SV}.
More precisely, in \emph{loc. cit.} the authors prove an equality of two maps between spaces of distributions on $p$-adic manifolds, which can be heuristically compared to the formula of theorem \ref{main thm} - taking into consideration the standard analogy between functional spaces and categories of sheaves. Indeed, one of the maps in \emph{loc. cit.} is a composition of two averaging operators, comparable to $\on{av}_{MN}^{M\bar N}\circ\on{Av}_{M\bar N}^{K}$. The other map is the \textbf{Bernstein map} (see \cite[\S 7]{BK} and \cite[\S 11]{SV}). This map is analogous to $L_{\gamma}$. Indeed, in \emph{loc. cit.} it is presented as some specialization to orbits at infinity of the wonderful compactification.
}

\subsection{
Second adjointness and Casselman's submodule theorem
}
Let $B$ be a Borel subgroup of $G$. 
\quash{The limit functor $L_{\gamma}$ has a natural upgrade to a functor \[L_{\gamma} : D_K(G/B)\ra D_{MN}(G/B)\]}
The limit functor can be interpreted as: $$ L_{\gamma} : D_K (G/B) \is D(X_1 / B) \to D (X_0 / B) \is D_{MN} (G/B).$$ In this form, it is the geometric Jacquet functor introduced by Emerton, Nadler and Vilonen in \cite{ENV} (in fact, they do not elaborate on the equivariancy that objects in the image of this functor enjoy, but see \cite{AM1,AM2}). Similarly to theorem \ref{main thm} above, we obtain a formula expressing the geometric Jacquet functor as a composition of two averaging functors 
(in fact, we prove it in a more general situation called 
the Matsuki setting, see theorem \ref{Thm:FormulaMatsuki}). Indeed, we actually use the result for $G/B$ in order to deduce the result for $G/K$.

As an application, we obtain (here, by $\calM_0 (\frakg , H)$ we denote the category of Harish-Chandra $(\frakg , H)$-modules with trivial infinitesimal character, and $J$ is the Casselman-Jacquet functor from \cite{ENV}, interpreted between the ``correct" equivariant categories):

\begin{corollary}(see theorem \ref{Second adjoint})\label{second adj}
The Casselman-Jacquet functor $$J : \calM_0 (\frakg , K) \to \calM_0 (\frakg , MN)$$ admits both a left and a right adjoint.
\end{corollary}

It is interesting to emphasize the resemblence of this corollary to the ``second adjointness" for the Jacquet functor from $p$-adic representation theory.

As another application of our formula we obtain a geometric proof of
\textbf{Casselman's submodule theorem}, based on a conservativity property of averaging (see \S\ref{Sec:CassSubmodThm})\footnote{The first geometric proof of Casselman's submodule theorem was given by Beilinson-Bernstein \cite{BB1}.}.

\quash{
\subsection{Future and related research}
Here we briefly indicate some of the possible directions of future research:
\begin{itemize}
\item In \cite{BFO}, the authors use the exactness property in corollary \ref{exactness for av} (in the group case)
together with other results to obtain a classification of character $\mathcal D$-modules on a complex
reductive group. One can try to combine their techniques with those of the present paper to obtain
a classification of character sheaves on symmetric spaces.

\item It would be very interesting to investigate the ``global" analogue of the
wonderful degeneration and study the corresponding nearby cycles functor.
For exmaple,
D. Gaitsgory and D. Nadler (\cite{GD}) use a version of the degeneration
of the symmetric space (in fact, they consider more general spherical varieties)
in the global setting to construct Langlands dual groups of symmetric spaces.
We expect that the technique in the present paper can be applied to the
global setting and we should obtain a global analogue of
our formula in theorem \ref{main thm} (see \cite{SS} where a similar question is studied in detail).
\end{itemize}}

\subsection{}
Here is a further outline of the contents and arguments in the paper.
\medskip

In section $2$, we introduce conventions, notations and recollections.

\medskip

In section $3$, we introduce the notion of a \textbf{limit scheme} and
define the main object of this paper, the \textbf{wonderful degeneration} of a symmetric space.
\medskip

In section $4$, we introduce and study the \textbf{limit functor} associated to a limit scheme.
We establish some basic functorial properties of limit functors.
We establish a key technical result of this paper, the \textbf{transversal fully faithfulness property} of
limit functors (see theorem \ref{TransRigid}), which says, roughly, that the limit functor is fully faithful on a pair of objects which are transversal. The results in this section are inspired by work of D. Nadler in \cite{N}.
\medskip

In section $5$, we introduce an algebraic version of \textbf{Matsuki datum} and the Matsuki correspondence from \cite{MUV}, as well as the resulting limit functor. We apply the transversal fully faithfulness property of section $4$ to prove a formula, which expresses the limit functor as a composition of an averaging functor with the Matsuki correspondence functor and use it to show that the 
limit functor admits both a left and a right adjoint  
(see theorem \ref{Thm:FormulaMatsuki} and corollary \ref{adjoint}).

\medskip

In section $6$, we specialize the setting of section $5$ to the case of the flag variety, and obtain a formula for the geometric Jacquet functor (see theorem \ref{FormulaForJC}). Using this formula, we 
establish a second adjointness theorem for Harish-Chandra modules 
(see theorem \ref{Second adjoint}) and reprove Casselman's submodule theorem, using a faithfullness property of averaging (see theorem \ref{CasselmanSubmoduleThm}).

\medskip

In section $7$, we recall the definition of character sheaves on symmetric spaces and we deduce our
main result, theorem \ref{main thm} (which is theorem \ref{formula for L} there), from
the results of sections $4$ and $6$.

\medskip

In appendix A, we prove that the wonderful degeneration $\tX$ introduced in section $3$ is
quasi-affine. This follows \cite[Appendix C]{DG}, and extends their result from the group case to symmetric spaces.
\medskip

\subsection{Acknowledgements}

A.Y.D. would like to thank his PhD advisor Joseph Bernstein, for suggesting to study the Casselman-Jacquet functor algebraically. Both authors would like to thank D. Gaitsgory for very helpful comments. Both authors would like to thank the Hausdorff Research Institute for Mathematics and Max Planck Institute for Mathematics, for excellent hosting and working conditions in the summer of 2014, during which the cooperation began. T.H.C. was partially supported by an AMS-Simons Travel Grant. A.Y.D. was partially supported by ERC grant 291612 and by ISF grant 533/14.


\section{Conventions, background and reminders} \label{Sec:Notation}

\subsection{General conventions} \label{Sec:NotationGeneral}

We fix an algebraically closed field $k$, and by scheme we mean a scheme of finite type over that field. By a variety we mean a separated scheme, and by a group we mean a smooth affine algebraic group. By $pt$ we denote the point. By $D(X)$ we denote one of the following:

\begin{enumerate}
	\item If $k = \bbC$, we can take $D(X)$ to be the derived category of constructible sheaves.
	\item If $k$ is of characteristic $0$, we can take $D(X)$ to be the derived category of holonomic $D$-modules.
	\item For $\ell \neq char(k)$, we can take $D(X)$ to be the derived category of constructible $\ell$-adic sheaves.
\end{enumerate}

We usually consider $D(X)$ as enriched over $D(pt)$, so that we automatically consider $Hom_{D(X)}(\cdot,\cdot) \in D(pt)$. By $C = C_X \in D(X)$ we denote the constant sheaf (i.e. $\pi^* C_{pt}$, where $\pi : X \to pt$), and by $\omega = \omega_X \in D(X)$ we denote the dualizing sheaf (i.e. $\pi^! C_{pt}$).

By $P(X) \subset D(X)$ we denote the full subcategory of perverse objects. We denote Verdier duality by $\bbD = \bbD_X : D(X) \to D(X)^{op}$.

\subsection{$\bbA^1$-schemes and nearby cycles}

For an $\bbA^1$-scheme $f : \wt{X} \to \bbA^1$, we denote by $X_t$ the fiber of $\widetilde{X}$ over $t \in \bbA^1$, and by $X^{\circ}$ the open subscheme $f^{-1}(\bG_m)$. For a morphism of $\bbA^1$-schemes $\phi : \wt{X} \to \wt{Y}$, we abuse notation and denote by $\pi$ also the corresponding base-changed morphisms $X_t \to Y_t$ and $X^{\circ} \to Y^{\circ}$.

We denote by $\psi_f : D(X^{\circ}) \to D(X_0)$ the (unshifted) functor of nearby cycles, so that $\psi_f [-1]$ is $t$-exact (w.r.t. the perverse $t$-structure).

\subsection{Symmetric spaces}\label{symmetric spaces}

When dealing with symmetric pairs in what follows, we will assume the following fixed data and notations.

Fix a connected reductive group $G$, and an (algebraic) involution $\theta:G \ra G$. Write as usual $G^{\theta}=\{g\in G \ | \ \theta(g)=g\}$. Fix an open subgroup $K$ of $G^{\theta}$. The pair $(G,K)$ is called a \emph{symmetric pair}, and $X := K \backslash G$ is the associated \emph{symmetric space}.

A torus $S \subset G$ is called $\theta$-\emph{split} if $\theta(s)=s^{-1}$ for all $s\in S$.
A parabolic subgroup $P \subset G$ is called $\theta$-\emph{split} if $L := P \cap \theta(P)$ is a Levi subgroup of $P$ (and of $\theta(P)$). Taking $A$ to be the maximal $\theta$-split torus in $Z(L)$, one has $L = Z_G (A)$.

Fix a minimal $\theta$-split parabolic $P \subset G$, and thus the corresponding $L := P \cap \theta(P)$ and $A$ - the maximal $\theta$-split torus in $Z(L)$. Note that $A$ is a maximal $\theta$-split torus in $G$, since $P$ is minimal.

Denote $M := L \cap K = Z_K(A)$, and denote $N := R_u (P)$. One has the decomposition $L = MA$, and thus the "Langlands decomposition" $P=MAN$. We also denote $\bar N = \theta(N)$ etc.

Fix also a co-character $\gamma : \bG_m \to A$, which we suppose to have negative pairing with roots of $N$.

Finally, we denote by $\mB$ the flag variety of $G$.

\quash{
We denote by $\fg=\Lie G$, $\fk=\Lie K$,
$\fp=\Lie P$, $\fn=\Lie N$, etc.
One has Cartan decompositions $\fg=\fk+\fraks$.
We fixed a $(G,\theta)$-invariant non-degenrate bilinear form $\langle,\rangle$ on $\fg$.

Let $T$ be a maximal torus of $G$ and
$B\supset T$ be a Borel subgroup. We denote by $U\subset B$
the unipotent radical of $B$.
Let $\mB=G/B$ denote the flag variety of $G$ and
$\mtB=G/U$ be the base affine variety.
We denote by $\bar U$ the unipotent radical
for the opposite Borel $\bar B$.

The following properties of orbits on $\mB$ are well-known:

\begin{lemma}\label{tansversal}

\begin{enumerate}
\item $MN$-orbits and $M\bar N$-orbits on $\mB$ are transversal.
\item $K$-orbits and $MN$-orbits (resp. $M\bar N$-orbits) on $\mB$ are transversal.
\end{enumerate}
\end{lemma}
}

\subsection{$G$-varieties and equivariant derived categories} \label{Sec:NotationGroups}

Let $X$ be a $G$-variety. As in \cite{BL}, one can define the equivariant derived category $D_G(X)$ as a 2-limit of categories $D(G \backslash Y)$, where $Y \to X$ runs over free $G$-resolutions of $X$. In fact, we can restrict ourselves to smooth resolutions, and smooth morphisms between them (cf. \cite[\S4]{BL}). If $K \subset G$ is a subgroup, then $D_K(X)$ can be realized as a 2-limit of categories $D(K \backslash Y)$, where $Y \to X$ runs over smooth free $G$-resolutions of $X$, and smooth morphisms between them.

Suppose that $H \subset G$ is a subgroup. Then the forgetful functor $\oblv^G_H : D_G(X) \to D_H(X)$ is conservative. In addition, it is fully faithful provided that $G / H$ is (connected) unipotent.

\subsection{Averaging functors}\label{averaging functor}

Let $X$ be a $G$-variety, and $H \subset G$ a subgroup. The forgetful functor $\oblv^G_H: D_G(X) \ra D_H(X)$ admits a left adjoint $\Ind_{G}^H:D_H(X) \ra D_G(X)$ and a right adjoint $\ind_{G}^H:D_H(X) \ra D_G(X)$.

Let $H_1,H_2 \subset G$ be two subgroups of $G$. We define the following averaging functors
\beq
\on{Av}^{H_1}_{H_2}:=\Ind_{H_2}^{H_1\cap H_2}\circ\oblv^{H_1}_{H_1\cap H_2}:
D_{H_1}(X)\ra D_{H_2}(X).
\eeq
\beq
\on{av}^{H_1}_{H_2}:=\ind_{H_2}^{H_1\cap H_2}\circ\oblv^{H_1}_{H_1\cap H_2}:
D_{H_1}(X)\ra D_{H_2}(X).
\eeq

\begin{lemma}\label{Av and oblv}
$\Av^{H_1}_{H_2}$ is left adjoint to $\av^{H_2}_{H_1}$.

\quash{
Assume $H_2=H_1\ltimes L$. Then we have a natural isomorphism of functors
\[\oblv_{L}^{H_2} \circ \Ind^{H_1}_{H_2} \is \Ind_{L}^{\{ e \} } \circ \oblv_{\{ e \} }^{H_1}.\]}
\end{lemma}

\subsection{Weakly monodromic sheaves}\label{sec_wm}

Let $X$ be a $G$-variety. Denote by $a,p : G \times X \to X$ the action and projection maps. We say that an object $\calF \in D(X)$ is \emph{weakly-equivariant}, if $a^* \calF \cong p^* \calF$ (an abstract isomorphism). We denote by $D_{G-wm} (X)$ the full subcategory of $D(X)$ generated under direct summands, cones, and shifts by the weakly-equivariant objects, and call the objects of $D_{G-wm} (X)$ \emph{weakly-monodromic}. Of course, the essential image of the forgetful functor $oblv^G_{\{ e \}} : D_G (X) \to D(X)$ lies in $D_{G-wm} (X)$.

More generally, for a scheme $\wt{X} \to B$ and a group scheme $\wt{G} \to B$, we define in a similar fashion $D_{\wt{G}-wm} (\wt{X})$, using $a,p : \wt{G} \times_B \wt{X} \to \wt{X}$.


\section{Limit schemes}\label{Sec:LimitSchemes}

\begin{definition}

A \emph{limit scheme} is an $\bbA^1$-scheme $f : \wt{X} \to \bbA^1$, equipped with a $\bbG_m$-action which $f$ intertwines with the standard $\bbG_m$-action on $\bbA^1$.

We also have in an obvious way the notions of morphisms between limit schemes, limit group schemes, and actions of limit group schemes on limit schemes.

\end{definition}

Note that for a limit scheme $\widetilde{X} \to \bA^1$, $X^{\circ}$ can be identified with $\bG_m \times X_1$ (a point $(t,x) \in \bG_m \times X_1$ is sent to $tx \in X^{\circ}$). Thus we can think of a limit scheme $f:\tX\ra\bA^1$ as a degeneration of $X_1$ to $X_0$.

\begin{example} [\emph{constant limit schemes}]\label{constant limit}

Let $X$ be a $\bG_m$-variety. We can consider $\widetilde{X} := \bA^1 \times X$, equipped with the diagonal $\bG_m$-action $t(s,x)=(ts,tx)$, and the projection $f: \widetilde{X} \to \bA^1$. We call it the \emph{constant limit scheme} associated to the $\bG_m$-variety $X$.

\end{example}

\subsection{Wonderful degenerations}\label{wonderful deg}

Let $G$ be a group, $K \subset G$ a subgroup, and $\gamma: \bG_m \to G$ a cocharacter.

Consider $G$ as a $\bG_m$-variety, via
$t\cdot g := \gamma(t)g$. Thus we have the corresponding constant limit scheme $\widetilde{G} (= \bA^1 \times G)$.

Consider the following closed subscheme of $G^{\circ} (= \bG_m \times G)$:
$$ K^{\circ} = \{ (t,g) | g \in \gamma(t)K\gamma(t)^{-1} \},$$
and let $\widetilde{K}$ be the closure of $K^{\circ}$ in $\widetilde{G}$. By \cite[proposition 2.3.8] {DG} or \cite{AM1}, $\widetilde{K} \ra \bA^1$ is a smooth subgroup scheme of $\wt{G} \to \bbA^1$. Consider the quotient $\widetilde{X} := \widetilde{K} \backslash \widetilde{G} \to \bbA^1$. By \cite{An}, $\widetilde{X}$ is a (separated) scheme of finite type over $\bA^1$. The $\bG_m$-action on $\widetilde{G}$ descends to $\widetilde{X}$, so that $\widetilde{X}$ is a limit scheme.

\quash{
\begin{remark}
Of course, one can analogously embed $\bG_m$ in $\bA^1 \cong \bG_m \cup \{ \infty \}$, and obtain $K_{\infty}$.
\end{remark}
}

We now assume that we are in the context of \S\ref{symmetric spaces}. In this case we have $K_0=MN$ (\cite{AM1}), hence the corresponding $\tX\ra\bA^1$ is a degeneration of the symmetric space $X=K\backslash G$ to $MN\backslash G$. One might call $\widetilde{X} \to \bA^1$ the \emph{wonderful degeneration}, because of its relation to the wonderful compactification.

We have the following nice proposition, whose proof is given in Appendix \ref{App:QuasiAff}; It extends the result in \cite[Appendix C]{DG} to symmetric spaces.

\begin{prop}
The quotient $\widetilde{X} := \widetilde{K} \backslash \widetilde{G}$ is a quasi-affine scheme.
\end{prop}

\begin{remark}

More generally, one can consider a co-character $\gamma:\bG_m\ra A$ which has non-positive pairing with roots of $N$.
Then $\gamma$ defines a $\theta$-split parabolic $P_1$ containing $P$ (it is $\bar N^{\gamma} P$ where $\bar N^{\gamma}$ is the centralizer of $Im(\gamma)$ in $\bar N$). One has the corresponding $M_1 = P_1 \cap K$ and $N_1 = R_u (P_1)$, and then $K_0 = M_1 N_1$.

So, in this case, one obtains a degeneration of the symmetric space $X = K \backslash G$ to $M_1 N_1 \backslash G$ (which is ``closer" to $K \backslash G$ than $MN \backslash G$).

\quash{
More generally, one can consider co-character  $\gamma:\bG_m\ra A$ which is non-positive on roots of $N$.
We now give a description of $K_0$ in this case.
Recall that a \emph{standard parabolic} of $G$ is a parabolic of $G$ containing $P$. There are in bijection with
subset of simple restricted roots $\Delta\subset\bX(A)$. Let $P_J$ be the standard parabolic corresponds to
$J\subset\Delta$, then it has a Langlands decomposition $P_J=M_JA_JN_J$. Here $A_J=\cap_{\alpha\in J}\ker\alpha$
and $M_JA_J=Z_G(A_J)$. We define $K_J:=K\cap M_J$.
Then using a similar method in [AM1], one can show that
$K_0\is K_JN_J$. Here $J=\{\alpha\in\Delta |\langle\gamma,\alpha\rangle=0\}$.
So in this case, the wonderful degeneration is a degeneration of the symmetric space $K\backslash G$ to
$K_JN_J\backslash G$.}
\end{remark}


\section{Limit functors}\label{Sec:LimitFunctor}

\subsection{Limit functors}

Let $f: \widetilde{X} \to \bA^1$ be a limit scheme. Note that $X^{\circ}$ can be identified with $\bG_m \times X_1$ (a point $(t,x) \in \bG_m \times X_1$ is sent to $tx \in X^{\circ}$). In particular, we have a projection $a_f: X^{\circ} \to X_1$.

\begin{remark}\label{G_m-equi}
We can identify $D(X_1)$ with $D_{\bG_m}(X^{\circ})$. Then $a_f^{*}$ is identified with the forgetful functor $D_{\bG_m}(X^{\circ}) \to D(X^{\circ})$.
\end{remark}

\begin{definition}
Let $f: \widetilde{X} \to \bA^1$ be a limit scheme. The \emph{limit functor} $L_f : D(X_1) \to D(X_0)$ is defined as $$L_f := \psi_f \circ a_f^{*}.$$
\end{definition}

\begin{remark}\label{ENV}
In particular, given a $\bG_m$-variety $X$, consider the associated constant limit scheme $f: \widetilde{X} \to \bA^1$. We have $X_t = X$ for every $t \in \bA^1$. Thus, in this case, the limit functor $L_f : D(X_1) \to D(X_0)$ is a functor $L_f : D(X) \to D(X)$. This functor appears in \cite{N}, and in \cite{ENV} for the special case of the flag variety.
\quash{We will sometimes write $L_{\alpha} := L_f$, where $\alpha : \bG_m \times X \to X$ denotes the $\bG_m$-action on $X$}
\end{remark}
\quash{
\begin{remark}{\color{red}Do we need this remark? Maybe remove it?}
By remark \ref{G_m-equi}, for any $\calF\in D(X_1)$, the sheaf $a_f^{*}(\calF)$ is $\bG_m$-equivariant on
$X^{\circ}$.
Therefore by  \cite[appendix]{AB},
the sheaf
$L_f(\calF)\in D(X_0)$ is $\bG_m$-monodromic. Moreover,
the monodromy action on $L_f(\calF)$ as a nearby cycles sheaf is opposite to the
$\bG_m$-monodromy action on it as a $\bG_m$-monodromic sheaf.
\end{remark}
}

\subsection{Basic properties}

The basic properties of the nearby cycles functor and the inverse image under a smooth morphism imply the following two lemmas:

\begin{lemma}\label{LimFunBasicProp} Let $f: \widetilde{X} \to \bA^1$ and $g: \widetilde{Y} \to \bA^1$ be two limis schemes, and $\phi : \widetilde{X} \to \widetilde{Y}$ a morphism of limit schemes. By abuse of notation, we will also denote by $\phi$ the morphisms $X_t \to Y_t$ and $X^{\circ} \to Y^{\circ}$ resulting from $\phi$ by base-change.
	
	\begin{enumerate}
		
		\item The functor $L_f$ is $t$-exact (w.r.t. the perverse $t$-structure) and commutes with 
		Verdier duality $\bbD$.
		
		\item There is a canonical morphism $\phi^* \circ L_g \to L_f \circ \phi^*$. Moreover, if
		$\phi$ is smooth then this morphism is an isomorphism.
		
		\item There is a canonical morphism $L_g \circ \phi_* \to  \phi_* \circ L_f$. Moreover, if
		$\phi$ is proper then this morphism is an isomorphism.
		
		\item The following diagrams commute:
		\[ \xymatrix{L_g \ar[r] \ar[d] & L_g \phi_*\phi^* \ar[d] \\
			\phi_*\phi^* L_g \ar[r] & \phi_* L_f \phi^* } \quad, \quad \xymatrix{L_f \ar[r] \ar[d] & L_f \phi^! \phi_! \ar[d] \\
			\phi^!\phi_! L_f \ar[r] & \phi^! L_g \phi_! }.\]
		
	\end{enumerate}
	
\end{lemma}

\begin{lemma}\label{lemma_BasicPropOfL}
	Let $f: \widetilde{X} \to \bA^1$ bet smooth a limit scheme.
	
	\begin{enumerate}
		\item The functor $L_f$ is (lax) monoidal\footnote{w.r.t. the standard ($*$-)monoidal structure $\otimes$.}, and strictly monoidal on the unit\footnote{i.e. the structural morphism $C_{X_0} \to L_f (C_{X_1})$ is an isomorphism.}.
		\item The functor $L_f$ commutes with Verdier duality $\bbD$ in the following precise sense: $$L_f (\omega_{X_1}) \cong \omega_{X_0}$$ and the structural morphism $$L_f (\UHom_{D(X_1)} ( - , \omega_{X_1})) \to \UHom_{D(X_0)} ( L_f (-) , L_f (\omega_{X_1}) )$$ is an isomorphism.
	\end{enumerate}
\end{lemma}

Another simple property is:

\begin{lemma}\label{Remark_LimitOfWM}
	Suppose that $f : \wt{X} \to \bbA^1$ is the constant limit scheme associated to a $\bbG_m$-variety $X$. Then the restriction of $L_f : D(X) \to D(X)$ to $D_{\bbG_m-wm} (X)$ is canonically isomorphic to the identity functor.
\end{lemma}

\subsection{Transversal fully faithfulness}

\begin{definition}

Let $\wt{X} \to B$ be a scheme, equipped with the action of two group schemes $\wt{K} , \wt{L} \to B$. The actions of $\wt{K}$ and $\wt{L}$ on $\wt{X}$ are said to be \emph{transversal}, if the morphism $$ a: \wt{K} \times_{B} \wt{L} \times_{B} \wt{X} \to \wt{X} \times_{B}  \wt{X}, \quad a(k,l,x):= (kx,lx)$$ is smooth. 

\end{definition}

\begin{remark}\label{Remark_MeaningOfTransversaility}
	If, in the setting of the previous definition, $\wt{X},\wt{K},\wt{L} \to B$ are all smooth, then the actions of $\wt{K}$ and $\wt{L}$ on $\wt{X}$ are transversal if and only if for every point $b$ of $B$ and point $x$ of $X_b$, one has $$T_x (K_b x) +T_x (L_b x) = T_x X_b,$$ i.e.``orbits are transversal".
\end{remark}

The purpose of this subsection is to prove the following theorem (for the convention regarding the subscript "$wm$", see \S\ref{sec_wm}):

\begin{thm}[\emph{Transversal fully faithfulness}]\label{TransRigid}
	Let $f: \wt{X} \to \bbA^1$ be a smooth limit scheme acted upon by two limit group schemes $\wt{K} , \wt{L} \to \bbA^1$. Assume that $f$ is proper, and that the actions of $\wt{K}$ and $\wt{L}$ on $\wt{X}$ are transversal. Then the natural morphism $$ Hom_{D(X_1)} (\calF , \calG) \to Hom_{D(X_0)} (L_f \calF , L_f \calG)$$ is an isomorphism, for all $\calF \in D_{K_1-wm} (X_1) , \calG \in D_{L_1-wm} (X_1)$.
\end{thm}

\begin{proof}
From lemma \ref{lemma_BasicPropOfL}, by some yoga of monoidal categories, it is enough to verify the following two properties:

\begin{enumerate}

\item The map $L_f (\calF) \otimes L_f (\calG) \to L_f (\calF \otimes \calG)$ is an isomorphism for $\calF,\calG$ as in the statement of the theorem.

\item The map $Hom_{D(X_1)} (C , \calH) \to Hom_{D(X_0)} (L_f(C) , L_f (\calH))$ is an isomorphism, for any $\calH \in D(X_1)$.

\end{enumerate}
Namely, we would have
\[\Hom_{D(X_1)} (\calF , \calG)\cong\Hom_{D(X_1)}(C,\bbD(\calF\otimes\bbD\calG))\stackrel{(2)}\cong
\Hom_{D(X_0)}(L_\gamma(C),L_\gamma(\bbD(\calF\otimes\bbD\calG)))\stackrel{(1)}\cong\]
\[\cong\Hom_{D(X_0)}(C,\bbD(L_\gamma\calF\otimes\bbD L_\gamma(\calG)))
\cong\Hom_{D(X_0)} (L_\gamma\calF ,L_\gamma\calG).\] and one can verify that the resulting isomorphism is the natural one.

The following two lemmas will confirm the above two properties.

\end{proof}

\begin{lemma}
	Let $f: \wt{X} \to \bbA^1$ be a scheme, acted upon by two group schemes $\wt{K}, \wt{L} \to \bbA^1$. Assume that the actions of $\wt{K}$ and $\wt{L}$ on $\wt{X}$ are transversal. Then the natural morphism $$\psi_f (\calF) \otimes \psi_f (\calG) \to \psi_f (\calF \otimes \calG)$$ is an isomorphism, for all $\calF \in D_{K^{\circ}-wm} (X^{\circ}) , \calG \in D_{L^{\circ}-wm} (X^{\circ})$.
\end{lemma}

\begin{proof}
	
	Let us define some morphisms: $$ \Delta: \wt{X} \to \wt{X} \times_{\bbA^1} \wt{X}, \quad \Delta(x):=(x,x),$$ $$ a: \wt{K} \times_{\bbA^1} \wt{L} \times_{\bbA^1} \wt{X} \to \wt{X} \times_{\bbA^1} \wt{X}, \quad a(k,l,x):=(kx,lx),$$ $$ p: \wt{K} \times_{\bbA^1} \wt{L} \times_{\bbA^1} \wt{X} \to \wt{X}, \quad p(k,l,x):=x.$$ $$ s: \wt{X} \to \wt{K} \times_{\bbA^1} \wt{L} \times_{\bbA^1} \wt{X}, \quad s(x):=(1,1,x). $$
	
	One immediately reduces to $\calF, \calG$ being weakly-equivariant (rather than just weakly-monodromic), so we assume that.
	
	What we need to show is that $\psi$ commutes with $\Delta^*$ on the object $\calF \boxtimes_{\bbG_m} \calG$. Since $\Delta = a \circ s$, it is enough to check that $\psi$ commutes with $a^*$, and that $\psi$ commutes with $s^*$ on $a^* (\calF \boxtimes_{\bbG_m} \calG)$.
	
	Regarding the first fact, it is simply because $a$ is smooth (by the transversality assumption).
	
	As for the second fact, since $\calF \boxtimes_{\bbG_m} \calG$ is $(K^{\circ} \times_{\bbG_m} L^{\circ})$-weakly-equivariant, we have $$a^* (\calF \boxtimes_{\bbG_m} \calG) \cong p^* \Delta^* (\calF \boxtimes_{\bbG_m} \calG).$$ Thus, it is enough to show that $\psi$ commutes with $s^*$ on the essential image of $p^*$. This is clear, since $\psi$ commutes with $p^*$ (since $p$ is smooth), and $\psi$ commutes with $(p \circ s)^*$ (since $p \circ s = id$ is smooth).
\end{proof}

\begin{lemma}
	Let $f: \wt{X} \to \bbA^1$ be a smooth limit scheme. Assume that $f$ is proper. Then the natural morphism $$ Hom_{D(X_1)} (C , \calH) \to Hom_{D(X_0)} (L_f(C) , L_f (\calH))$$ is an isomorphism, for any $\calH \in D (X_1)$.
\end{lemma}

\begin{proof}
Denote by $g: \bA^1 \to \bA^1$ the constant limit scheme (i.e. $g = id$ and the $\bG_m$-action is the standard one) and by $\pi: \widetilde{X} \to \bA^1$ the projection (i.e. $\pi = f$).
Then the morphism $$Hom_{D(X_1)} (C , \calH) \to Hom_{D(X_0)} (L_f (C) , L_f (\calH))$$ can be expressed as $$ Hom_{D(X_1)} (C , \calH) \cong \pi_* \calH \cong L_g (\pi_* \calH) \cong \pi_* L_f (\calH) \cong Hom_{D(X_0)}(C , L_f (\calH))\cong$$
$$
\cong Hom_{D(X_0)} (L_f (C) , L_f (\calH)) .$$
The lemma follows.

\end{proof}

\subsection{Equivariant limit functors}\label{equivariant limit}

Let $\widetilde{X}$ be a limit scheme acted upon by a limit group scheme $\widetilde{K}$. One can try  to consider the "limit stack" $\widetilde{K} \backslash \widetilde{X}$ and the corresponding limit functor $D_{K_1}(X_1) \to D_{K_0}(X_0)$. However, we will not try to deal with such generalities, and restrict ourselves to a special case (the one that we will need).

Let $G$ be a group, and $\gamma : \bG_m \to G$ a cocharacter. Let $X$ be a $G$-variety. Via $\gamma$, $X$ is a $\bG_m$-variety, and we consider the associated constant limit scheme $f : \widetilde{X} \to \bA^1$ and the resulting limit functor $L_{\gamma} := L_f : D(X) \to D(X)$. We also consider the constant limit group scheme $\wt{G}$ associated to $G$ equipped with the $\bbG_m$-action $t \cdot g := \gamma (t) g \gamma(t)^{-1}$. Then $\wt{G}$ acts on $\wt{X}$ as limit schemes.

Let $K \subset G$ be a subgroup. We have the corresponding subgroup limit scheme $\wt{K} \subset \wt{G}$ (as in \S\ref{wonderful deg}).

\begin{constructionlemma}\label{LemmaEquiLift}
The functor  $L_{\gamma}: D(X) \to D(X)$ admits a natural lifting to a functor $L_{\gamma} : D_K(X) \to D_{K_0}(X)$ (which we denote by the same name, by abuse of notation).
\end{constructionlemma}

\begin{proof}

Step 1: Recall that the equivariant derived category is a $2$-limit of certain approximations, obtained by free resolutions. So let us assume first that $X$ is a free $G$-variety. Then the stack $\widetilde{K} \backslash \widetilde{X}$ is a scheme; Indeed, in general it can be written $\left( \widetilde{K} \backslash \widetilde{G} \right) \times_{\widetilde{G}} \widetilde{X}$. As noted in \S\ref{wonderful deg}, $\widetilde{K} \backslash \widetilde{G}$ is a scheme, and since $X$ is a free $G$-variety, locally, writing $X = G \times Y$, we will have $\left( \widetilde{K} \backslash \widetilde{G} \right) \times \widetilde{Y}$. To conclude, in the case that $X$ is a free $G$-variety, $\widetilde{K} \backslash \widetilde{X} \to \bA^1$ is a limit scheme, whose fiber over $t \in \bA^1$ is $K_t \backslash X$. Applying the limit functor to this limit scheme, we obtain $L : D(K \backslash X) \to D(K_0 \backslash X)$, which is the same as $L: D_K(X) \to D_{K_0}(X)$.

Step 2: In general, as recalled in \S\ref{Sec:NotationGroups}, objects of $D_K(X)$ are represented by objects of $D(K \backslash Y)$ for various smooth free $G$-resolutions $Y \to X$, satisfying compatibility with smooth pullback (an similarly for $D_{K_0}(X)$ etc.). Since the limit functor commutes with smooth base change, it is clear that we can "glue" the limit functors of step 1 into a limit functor $D_K(X) \to D_{K_0}(X)$.

\end{proof}

\begin{remark}\label{Rem:EquiLimFuncComm}
As a particular case of the above construction, suppose that $X$ is a $\bG_m$-variety, equipped with an action of a group $M$, and that the $\bG_m$-action on $X$ and the $M$-action on $X$ commute. Then (say, applying the above for $G := \bG_m \times M$ and $K := M$), we obtain a lifting of the limit functor $L : D(X) \to D(X)$ to a functor $L : D_M(X) \to D_M(X)$ (by abuse of notation, we will denote it by the same name).
\end{remark}

Let $L \subset G$ be another subgroup, and $M \subset K \cap L$ a subgroup which lies in the centralizer of $Im(\gamma)$. We have an equivariant version of theorem \ref{TransRigid}:

\begin{thm}\label{TransRigidEqui}
	The natural morphism $$ Hom_{D_M(X)} (\calF , \calG) \to Hom_{D_M (X)} (L_{\gamma} (\calF) , L_{\gamma} (\calG))$$ is an isomorphism, for all $\calF \in D_{K-wm}(X), \calG \in D_{L-wm} (X)$.
\end{thm}

\begin{proof}

We consider the ``enriched" $Hom$, \[ HOM_{D_M(X)} (\cdot , \cdot) \in D_M(pt) \] (one has $HOM_{D_M(X)}(\calF,\calG) \cong \bbD (\bbD \calG \otimes \calF)$). Applying the ``invariants" functor $D_M(pt) \to D(pt)$ to $HOM_{D_M(X)} (\cdot , \cdot)$ one obtains \[ Hom_{D_M(X)}(\cdot,\cdot), \] while applying the forgetful functor $D_M(pt) \to D(pt)$ to $HOM_{D_M(X)} (\cdot , \cdot)$ one obtains \[ Hom_{D(X)}(\oblv^M_{\{ e \}}(\cdot),\oblv^M_{\{ e \}}(\cdot)).\]

By these remarks, for objects $\calF,\calG \in D_M(X)$, the map \[ Hom_{D_M(X)} (\calF,\calG) \to Hom_{D_M(X)} (L_{\gamma} \calF, L_{\gamma} \calG) \] is an isomorphism if the map \[ HOM_{D_M(X)} (\calF,\calG) \to HOM_{D_M(X)} (L_{\gamma} \calF, L_{\gamma} \calG) \] is an isomorphism. Since the forgetful functor $D_M(pt) \to D(pt)$ is conservative, the later map is an isomorphism if the map \[ Hom_{D(X)}(\oblv^M_{\{ e \}}(\calF),\oblv^M_{\{ e \}}(\calG)) \to Hom_{D(X)}(L \oblv^M_{\{ e \}}(\calF),L \oblv^M_{\{ e \}}(\calG)) \] is an isomorphism. But this map is an isomorphism indeed, by theorem \ref{TransRigid}.
\end{proof}

We will also need the following ``relative" version of theorem \ref{TransRigidEqui}:

\begin{thm}\label{relative version}
We preserve the set up in theorem \ref{TransRigidEqui}. Let $Z$ be another smooth variety. Regard $X\times Z$ as a $\bG_m$-variety and an $M$-variety where $\bG_m$ and $M$ act only on $X$.
Let $\calF=\calF_1\boxtimes\calF_2,\calG=\calG_1\boxtimes\calG_2 \in D_{M}(X\times Z)$ and suppose that
$\calF_1 \in D_{K-wm}(X), \calG_2 \in D_{L-wm} (X)$.
Then the natural morphism
$Hom_{D_{M}(X\times Z)} (\calF,\calG) \to Hom_{D_{M}(X\times Z)}(L_\gamma \calF, L_\gamma \calG)$ is an isomorphism.
\end{thm}
\begin{proof}
By the same argument as in theorem \ref{TransRigidEqui}, it suffices to prove the non-equivariant version, so let us suppose that $M=\{ e \}$.
Since $\bG_m$ acts only on $X$, we have a natural isomorphism
$L_\gamma(\calF)\is L_\gamma(\calF_1)\boxtimes\calF_2$,  
$L_\gamma(\calG)\is L_\gamma(\calG_1)\boxtimes\calG_2$
and the
K$\ddot{\on{u}}$nneth formula implies
\[Hom_{D(X\times Z)} (\calF,\calG)\is Hom_{D(X)}(\calF_1,\calG_1)\otimes Hom_{D(Z)}(\calF_2,\calG_2)\]
and
\[Hom_{D(X\times Z)} (L_\gamma\calF,L_\gamma\calG)\is Hom_{D(X)}(L_\gamma\calF_1,L_\gamma\calG_1)\otimes Hom_{D(Z)}(\calF_2,\calG_2).\]
Moreover, under the isomorpshims above the map
\[ Hom_{D(X\times Z)} (\calF,\calG) \to Hom_{D(X\times Z)}(L_\gamma \calF, L_\gamma \calG) \]
becomes
\[Hom_{D(X)}(\calF_1,\calG_1)\otimes Hom_{D(Y)}(\calF_2,\calG_2)\stackrel{h\otimes id}\ra Hom_{D(X)}(L_\gamma\calF_1,L_\gamma\calG_1)\otimes Hom_{D(Z)}(\calF_2,\calG_2),\]
where $h$ is an isomorphism by theorem \ref{TransRigid}. The result follows.

\end{proof}

\section{A formula in the Matsuki setting}\label{Matsuki setting}

In this section we obtain a formula, which describes the limit functor as a composition of two averaging functors, in a setting similar to that of the Matsuki datum of \cite{MUV}.

\begin{definition}\label{Matsuki datum}
A \emph{Matsuki datum} consists of the following. We are given:
\begin{itemize}

\item A group $G$.

\item Two closed subgroups $H_{0} , H_{\infty} \subset G$. Denote $M := H_{0} \cap H_{{\infty}}$.

\item A co-character $\gamma : \bG_m \to G$.

\item A smooth $G$-variety $X$, which we also consider as a $\bG_m$-variety via $\gamma$.

\end{itemize}

We suppose that:

\begin{itemize}

\item $H_{0} / M$ and $H_{\infty} / M$ are (connected) unipotent groups.

\item $Im(\gamma)$ normalizes $H_{0}$ and $H_{\infty}$, and centralizes $M$.

\item Each $H_{0}$-orbit in $X$ is transversal to each $H_{\infty}$-orbit in $X$.

\item The action of $\bG_m$ on $X$ preserves the $H_{0}$-orbits and the $H_{\infty}$-orbits.

\item $M$ has finitely many orbits in the $\bG_m$-fixed point set $X^{\bG_m}$.

\item For any $M$-orbit $\mO$ in $X^{\bG_m}$ denote
\[\mO_{0}=\{x\in X \ | \ \lim_{t\ra 0} t\cdot x\in\mO\},\ \ \
\mO_{\infty}=\{x\in X \ | \ \lim_{t\ra\infty} t\cdot x\in\mO\}\footnote{Here $t\in\bG_m$ and we embed $\bG_m$ into the projective line $\mathbb{P}^1$
so that $\mathbb{P}^1 - \bG_m=\{0,\infty\}$.}.\]
Then $\mO_0$ (resp. $\mO_\infty$) is a single $H_0$-orbit (resp. $H_\infty$-orbit), and
the correspondence $\mO_0\leftrightarrow\mO_\infty$ is a bijection between
$H_0$-orbits and $H_\infty$-orbits in $X$.

\end{itemize}

\end{definition}

We will denote by $(G , H_{0}, H_{\infty} , \gamma , X)$ a datum as above.

In analogy with theorem 5.3 of \cite{MUV}, we have:

\begin{thm}\label{Thm:AvIsEqui}
The adjunction between $\Av_{H_{\infty}}^{H_{0}}$ and $\av_{H_{0}}^{H_{\infty}}$ defines an equivalence between $D_{H_{0}}(X)$ and $D_{H_{\infty}}(X)$.
\end{thm}

\begin{proof}
The following proof was hinted to us by D. Gaitsgory. It uses Braden's hyperbolic localization theorem; for details see \cite{B}, \cite{DG2}. We want to show that the unit morphism $$ id \to av^{H_{\infty}}_{H_0} Av^{H_0}_{H_{\infty}}$$ is an isomorphism (the counit morphism is dealt with analogously). Fixing an $M$-orbit $\calO$ in $X^{\bbG_m}$, we consider the diagram $$ \xymatrix{\calO \ar@/_0.5pc/[r]_{i_{\infty}} \ar@/^0.5pc/[d]^{i_0} & \calO_{\infty} \ar[d]^{j_{\infty}} \ar@/_0.5pc/[l]_{\pi_{\infty}} \\ \calO_0 \ar[r]^{j_0} \ar@/^0.5pc/[u]^{\pi_0} & X}$$ (where $i,j$ are the inclusions maps and $\pi$ are the contraction maps). Then by Braden's theorem, on the subcategory $D_{\bbG_m-wm}(X) \subset D(X)$, one has isomoprhisms of functors $$ i_0^* j_0^! \cong (\pi_0)_*  j_0^! \cong (\pi_{\infty})_!  j_{\infty}^* \cong i_{\infty}^! j_{\infty}^*;$$ Let us denote them by $Loc_{\calO}$.

We will omit a lot of forgetful functors in what follows, to make it more readable; The reader should be able to figure them out. To show that our unit map is an isomorphism, it is enough to show that it is so after application of $Loc_{\calO}$, for every $\calO$. We have a commutative diagram $$ \xymatrix{id \ar[r] \ar[rd] & av^{H_{\infty}}_{H_0}\circ Av^{H_0}_{H_{\infty}} \ar[d] \\ & Av^{H_0}_{H_{\infty}}}.$$ Thus, to show that the upper arrow is an isomorphism after application of $Loc_{\calO}$, it is enough to show that the other two are so. We are reduced to showing that $id \to oblv^{H_{\infty}}_M \circ Ind^M_{H_{\infty}}$ and $oblv^{H_0}_M \circ ind^M_{H_0}\to id $ are isomorphisms after application of $Loc_{\calO}$, on $\bbG_m$-weakly-monodromic objects. For the first morphism it is easy using the third description of $Loc_{\calO}$, while for the second morphism it is easy using the second description of $Loc_{\calO}$.

\end{proof}

Considering the constant limit scheme associated to the $\bG_m$-variety $X$, and the $M$-action commuting with the $\bG_m$-action, we obtain a limit functor $$L_{\gamma} : D_M(X) \to D_M(X)$$ as in remark \ref{Rem:EquiLimFuncComm}.

\begin{definition}

A subgroup $K \subset G$ is called \emph{adapted} to the Matsuki datum, if $K \cap H_0 = K \cap H_{\infty} = M$, $K_0 = H_0$ (for the notation $K_0$ see \S\ref{wonderful deg}), $K$ has finitely many orbits on $X$, and each $K$-orbit is transversal to each $H_{\infty}$-orbit in $X$.

\end{definition}

\begin{remark}\label{Rem:AdaptedIsLimitGood}
If $K$ is adapted to the Matsuki datum, then (by remark \ref{Remark_MeaningOfTransversaility}) the actions of $\wt{K}$ and $\wt{H_{\infty}}$ on $\wt{X}$ are transversal. Here, $\wt{X}$ is the constant limit scheme associated to $X$, and $\wt{K} , \wt{H_{\infty}}$ are as in \S\ref{wonderful deg} (notice that $\wt{H_{\infty}}$ is just constant, since $Im(\gamma)$ normalizes $H_{\infty}$).
\end{remark}

From now on, let $K \subset G$ be a subgroup adapted to the Matsuki datum. Recall that the limit functor $L_{\gamma} : D_M (X) \to D_M (X)$ from above lifts to a limit functor $$ L_{\gamma} : D_K (X) \to D_{H_0} (X).$$

\begin{thm}\label{Thm:FormulaMatsuki}
Suppose that $X$ is proper. Then one has an isomorphism of functors $D_K(X) \to D_{H_{0}}(X)$:
\[ L_{\gamma}  \is\av_{H_{0}}^{H_{\infty}} \circ\Av_{H_{\infty}}^{K} .\]
\end{thm}
\begin{proof}
Consider $\calF \in D_K(X)$ and $\calG \in D_{H_{\infty}}(X)$. We have:
\beqn
\xymatrix@C=1em{
Hom_{D_{H_{\infty}}(X)} (\Av^{K}_{H_{\infty}} \calF , \calG) \ar[r]^-{\sim} & Hom_{D_{M}(X)} ( \calF , \calG) \ar[r]^-{\sim} & Hom_{D_{M}(X)} (L_{\gamma} \calF, L_{\gamma} \calG) \ar[r]^-{\sim} &
}
\eeqn
\beqn
\xymatrix@C=1em{
 \ar[r]^-{\sim} & Hom_{D_{M}(X)} (L_{\gamma} \calF, \calG) \ar[r]^-{\sim} & Hom_{D_{H_{\infty}}(X)} (\Av^{H_{0}}_{H_{\infty}} L_{\gamma} \calF , \calG )
}
\eeqn

Here, the first isomorphism is by adjunction, the second isomorphism is by theorem \ref{TransRigidEqui}, the third isomorphism is due to remark \ref{Remark_LimitOfWM} and noticing that $H_{\infty}$-equivariant sheaves are $\bbG_m$-weakly-monodromic, and the fourth isomorphism is by adjunction again.

From this, by Yoneda lemma, we obtain an isomorphism of functors $D_K(X) \to D_{H_{\infty}}(X)$:
$$
\Av^{H_{0}}_{H_{\infty}} \circ L_{\gamma} \cong\Av^{K}_{H_{\infty}}.
$$

Applying $\av_{H_{0}}^{H_{\infty}} \circ \cdot$ to both sides, and taking into account theorem \ref{Thm:AvIsEqui}, we obtain an isomorphism of functors $D_K(X) \to D_{H_{0}}(X)$:

$$
L_{\gamma} \cong\av_{H_{0}}^{H_{\infty}} \circ\Av^{M}_{H_{\infty}}.
$$

\end{proof}
\begin{corollary}\label{adjoint}
The functor $L_{\gamma}$ admits both a left and a right adjoint.
\end{corollary}
\begin{proof}
The fact that $L_{\gamma}$ commutes with Verdier duality implies  
\[L_{\gamma}\is\av_{H_{0}}^{H_{\infty}} \circ\Av^{M}_{H_{\infty}}\is
\Av_{H_{0}}^{H_{\infty}} \circ\av^{M}_{H_{\infty}}.\]
Since $\av_{H_{0}}^{H_{\infty}}$, $\Av_{H_{0}}^{H_{\infty}}$
are equivalences of categories (by theorem \ref{Thm:AvIsEqui}) and $\Av^{M}_{H_{\infty}}$ (resp. $\av^{M}_{H_{\infty}}$)
admits a right adjoint (resp. left adjoint), the corollary follows.
\end{proof}


\section{Casselman-Jacquet functor}\label{Sec:CJFunctor}

In this section we obtain a formula, which describes the geometric Jacquet functor of \cite{ENV} as a composition of two averaging functors.

Suppose that we are in the context of \S\ref{symmetric spaces}.

\begin{lemma}\label{Lem:FlagIsMatsuki}
The datum $(G , MN , M\bar N , \gamma , \mB)$ is a Matsuki datum. The subgroup $K \subset G$ is adapted to this Matsuki datum.
\end{lemma}

\begin{proof}

Choosing a maximal torus $A \subset T \subset P$, the set of $T$-fixed points $\mB^T$ is finite, and by Bruhat decomposition, any $MN$-orbit (and $M \bar N$-orbit) contains one of them. From this one sees that $MN$-orbits coincide with $MAN=P$-orbits in $\mB$. Also, we easily see that $\mB^{\gamma} = M \cdot \mB^T$ ($\mB^{\gamma}$ denotes the set of points in $\mB$ fixed by $Im(\gamma)$). Finally, for $x \in \mB^T$, $m \in M$ and $n \in N$, one has $\lim_{t \to 0} \gamma(t) nmx = mx$ (and analogously for $\bar N$).

From these remarks, it is easy to verify the six demands of \ref{Matsuki datum}.

To show that $K$ is adapted to the Matsuki datum; That $K_0 = MN$ was recalled in \S\ref{wonderful deg}. The last two requirements in the definition of an adapted subgroup follow from the Iwasawa decomposition.

\end{proof}

Thus, we obtain the following theorem, as a consequence of theorem \ref{Thm:FormulaMatsuki} and corollary \ref{adjoint}:

\begin{thm}\label{FormulaForJC}
\
\begin{enumerate}
\item
One has the following isomorphism of functors $D_K(\mB) \to D_{MN}(\mB)$:
\[L_{\gamma} \cong\av^{M\bar N}_{MN} \circ\Av^K_{M\bar N}.\]
\item
The geometric Jacquet functor $L_{\gamma}$ admits both a
left and a right adjoint.
\end{enumerate}
\end{thm}

\begin{remark}
Here, $L_{\gamma}$ can be understood in two isomorphic ways. The first, as a composition of $\oblv^K_M$ with the limit functor $L_{\gamma}: D_M(\mB) \to D_M(\mB)$ (and the image in fact lies in the full subcategory $D_{MN}(\mB)$) - this is the approach of \S\ref{Matsuki datum}. The second, as the equivariant limit functor $D_K(\mB) \to D_{MN}(\mB)$ arising from construction-lemma \ref{LemmaEquiLift}.
\end{remark}

\subsection{Applications: Second adjointness and Casselman's submodule theorem}\label{Sec:CassSubmodThm}
In this section we work in the $D$-modules setting.
We give applications of theorem \ref{FormulaForJC}; Namely, we 
establish a second adjointness theorem for Harish-Chandra modules and 
reprove Casselman's submodule theorem.

\subsubsection{
Second adjointness 
}

We suppose that we are in the setting of \S\ref{symmetric spaces}. Denote by $\frakg$ the Lie algebra of $G$, etc. 
For any subgroup $H$ of $G$,
let $\mathcal M_0(\fg, H)$
denote the category of Harish-Chandra $(\fg,H)$-modules with trivial infinitesimal character.
Consider the Casselman-Jacquet functor 
$J:\mathcal M_0(\fg,K)\ra\mathcal M_0(\fg,MN)$, as in, say, \cite{ENV}. 
Recall that $J(V)$ consists of $\fn$-finite vectors inside the completion of $V$ w.r.t. powers of $\bar\frakn$.

\begin{thm}\label{Second adjoint}
The Casselman-Jacquet functor $$J : \calM_0 (\frakg , K) \to \calM_0 (\frakg , MN)$$ admits both a left and a right adjoint. 
\end{thm}
\begin{proof}
It follows from theorem \ref{FormulaForJC} that the geometric Jacquet functor 
$L_\gamma:P_K(\mB)\ra P_{MN}(\mB)$, viewed as a functor 
between abelian categories of equivariant perverse sheaves, admits 
both a left and a right adjoint. Now the 
theorem follows from \cite[theorem 1.1]{ENV}
and the Riemann-Hilbert correspondence.
\end{proof}
\begin{remark}
We could replace $\mathcal M_0(\fg,K)$ by Harish-Chandra $(\fg,K)$-modules with
arbitrary generalized infinitesimal character (although some care is required for a non-regular infinitesimal character). 
Also, it follows from the proof of the theorem above that the left and right adjoint 
of $J$, under the localization theorem, are given by zero cohomologies 
of composition of 
two averaging functors.
It would be desirable to have explicit algebraic formulas for
the adjoint functors.
This remark will be elaborated on in a future paper.

\end{remark}

\subsubsection{Casselman's submodule theorem}
Recall the statement of Casselman's submodule theorem:
\begin{thm} \label{CasselmanSubmoduleThm}

Let $V$ be a finitely-generated $(\fg , K)$-module. Assume that $V \neq 0$. Then $V / \frakn V \neq 0$.

\end{thm}

\quash{
\begin{remark}
Note that because of $K$-equivariancy, it is equivalent to demand in the theorem that for some $x$ in the open $K$-orbit one has $M / \frakn_x M \neq 0$.
\end{remark}
}

Let us skecth now the translation of this problem to a geometric one. One can test whether $V / \frakn V \neq 0$ by checking that the Casselman-Jacquet module of $V$ is non-zero. 
It is easy (see \cite{BB1}) to reduce the statement to the case where $V$ has a regular central character. By \cite{ENV}, in that case we can pass (under Beilinson-Bernstein equivalence) to the picture of twisted $D$-modules, and the limit functor $L_{\gamma} : D_{\lambda}(\mB) \to D_{\lambda}(\mB)$ realizes the Casselman-Jacquet functor ($\lambda$ is some twisting, suited to the central character). We suppose for simplicity that $\lambda$ corresponds just to usual $D$-modules. In this way, theorem \ref{CasselmanSubmoduleThm} is reduced to the following statement:

\begin{thm}

The functor $L_{\gamma} : D_K(\mB) \to D_{MN}(\mB)$ is conservative (i.e., sends non-zero objects to non-zero objects).

\end{thm}

By theorem \ref{FormulaForJC}, and the fact that $\av^{M\bar N}_{MN}$ is an equivalence, this last theorem is reduced to checking that $\Av^K_{MN}$ is conservative. This follows from the following proposition:

\begin{proposition}

Let $U \subset G$ be maximal unipotent. Let $\calF \in D(\mB)$, and suppose that $\calF$ is constructible w.r.t. a stratification which is transversal to the stratification by $U$-orbits. Assume that $\calF \neq 0$. Then $\Av^{\{ e \}}_U (\calF) \neq 0$.

\end{proposition}

\begin{proof}

Denote by $S$ a stratification as in the statement of the theorem; i.e. $\calF$ is constructible w.r.t. $S$, and $S$ is transversal to the stratification by $U$-orbits. Let $S_{\calF} = \{ A \in S | i_A^* \calF \neq 0 \}$ (here $i_A : A \to X$ is the inclusion). By lemma \ref{TransversalStratificationLemma} below, there exists an $U$-orbit $B$ such that $B \cap (\cup_{A \in S_{\calF}} A)$ is finite and non-empty. Write $i: B \to X$ for the inclusion and $\pi : B \to pt$ for the map to the point. We have that $i^* \calF$ is non-zero, and supported in finitely many points. This easily gives that $\pi_! i^* \calF \neq 0$. Then $Hom ( \pi_! i^* \calF , C_{pt} ) \neq 0$ and since:
$$ Hom_{D_U(\mB)} ( \Av^{\{ e \}}_U (\calF) , i_* \pi^! C_{pt} ) = Hom_{D(\mB)}( \calF , i_* \pi^! C_{pt} ) = Hom_{D(pt)} ( \pi_! i^* \calF , C_{pt} ) $$
we get obviously $\Av^{\{ e \}}_U (\calF) \neq 0$.

\end{proof}

\begin{lemma} \label{TransversalStratificationLemma}
Let $X$ be a variety, and $S$,$T$ two transversal stratifications of $X$. Assume:
\begin{itemize}
\item $X$ is connected, smooth and proper.
\item All strata of $T$ are affine.
\end{itemize}
Then for every $A \in S$, there exists $B \in T$ such that $codim(B)=dim(A)$ and $A \cap B \neq \emptyset$.

\end{lemma}
\begin{proof}

The proof is by induction on $dim(A)$. If $dim(A)=0$, the lemma is clear (the unique open stratum of $T$ will do the job).

So let $A \in S$, $dim(A) \ge 1$.

Let $B \in T$ satisfy $dim(\partial A) \leq codim(B) < dim(A)$ (the dimension of an empty variety is $-1$). We will show now that if $A \cap B \neq \emptyset$, then we can find $B_1 \in T$ such that $codim(B) < codim(B_1)$ and $A \cap B_1 \neq \emptyset$. Indeed, $\bar{A} \cap \bar{B} \neq \emptyset$. But as $\bar{A} \cap \bar{B}$ is proper and of dimension $\ge 1$, it is not possible that $\bar{A} \cap \bar{B} \subset B$ (Since $B$ is affine). Hence, $\bar{A} \cap \partial B \neq \emptyset$. But notice that $\partial A \cap \partial B = \emptyset$, because of the condition $dim(\partial A) \leq codim(B)$. Hence, $A \cap \partial B \neq 0$. Thus we know that $A$ intersects with a $T$-stratum of bigger codimension.

This argument shows that if we will find $B \in T$ intersecting $A$ and such that $dim(\partial A) \leq codim(B)$, we are done. If $\partial A = \emptyset$, this is trivial. If $\partial A \neq \emptyset$, consider a stratum of $S$, $A_1 \subset \partial A$ of largest possible dimension. By induction, there is $B_1 \in T$ such that $ A_1 \cap B_1 \neq \emptyset$ and $codim(B_1)=dim(A_1)$. But by standard algebraic geomtery, since $A \cup A_1$ is an irreducible subvariety of dimension bigger than $dim(A_1)$, we see that $dim ((A \cup A_1) \cap B_1) \ge 1$, and hence (since $dim (A_1 \cap B_1)  = 0$) we conclude $A \cap B_1 \neq \emptyset$.

\end{proof}


\section{The formula for the wonderful degeneration}\label{formula for WD}

In this section we prove the main result of this paper which says that, on the category of character sheaves on
the symmetric space $X$, the limit functor associated to the wonderful degeneration of $X$ is isomorphic to a
composition of two averaging functors.

We suppose throughout this section that we are in the setting of \S\ref{symmetric spaces}.

\subsection{The limit functor}

Recall that we have the limit functor \[L_{\gamma} : D_{K\times K}(G)\is D_K(G/K) \to D_{MN}(G/K)\is D_{MN\times K}(G).\] 
Let us remind that the functor $L_{\gamma} : D_{K\times K}(G) \to D_{MN\times K}(G)$ is described as follows (identifying $D_{K\times K}(G)$ with $D_K(G/K)$ and $D_{MN\times K}(G)$ with $D_{MN}(G/K)$): one 
applies construction-lemma \ref{LemmaEquiLift} for the equivariant limit functor 
to the case of the constant limit scheme $f: \widetilde{X} := G/K\times\bA^1 \to \bA^1$
and the limit group scheme $\tK$ whose fiber over $1$ is $K$, and whose fiber over $0$ is $MN$.

\quash{
\begin{remark}
Let us remind that the functor $L_{\gamma} : D_{K\times K}(G) \to D_{MN\times K}(G)$ is described as follows (identifying $D_K(G)$ with $D(X)$ and $D_{MN}(G)$ with $D(MN \backslash G)$)
: one has the limit scheme $f: \widetilde{X} := \widetilde{K} \backslash \widetilde{G} \to \bA^1$, whose fiber over $1$ is $X = K \backslash G$, and whose fiber over $0$ is $MN \backslash G$. Then $L_{\gamma}$ is the composition of two functors: $a_f^*$ which "spreads" $\bG_m$-equivariantly a sheaf on $f^{-1}(1) = X$ to a sheaf on $f^{-1}(\bG_m)$, and $\psi_f$, the nearby cycles functor.
\end{remark}}

\subsection{Character sheaves}\label{Sec:AdmShv}

\subsubsection{Horocycle correspondence}
Consider the Horocycle correspondence
\[G\stackrel{q}\la G\times\mB\stackrel{p}\ra \mB\times\mB,\]
where $q(g,B)=g$, $p(g,B)=(gBg^{-1},B)$.
The map $q$ is proper and smooth, and $p$ is smooth.
Notice that $G\times G$ acts on $G$, $G\times\mB$
and $\mB^2:=\mB\times\mB$ by the formulas $(g_1,g_2)\cdot g=g_1gg_2^{-1}$,
$(g_1,g_2)\cdot (g,B)=(g_1gg_2^{-1},g_2Bg_2^{-1})$ and
$(g_1,g_2)\cdot (B_1,B_2)=(g_1B_1 g_1^{-1},g_2B_2 g^{-1}_2)$. Moreover, the maps
$p$, $q$ are $G\times G$-equivariant w.r.t. these actions.
We define $CH=q_!p^*:D_{K\times K}(\mB^2)\ra D_{K\times K}(G)$ and its right adjoint $HC=p_*q^!:D_{K\times K}(G)\ra D_{K\times K}(\mB^2)$. 

We consider all the above varieties also as $G$-varieties, via the morphism $G \to G\times G$, $g \mapsto (g,e)$. Thus, 
for subgroups $H_1,H_2 \subset G$ we write $\Av^{H_1}_{H_2}$ for $\Av^{H_1\times\{e\}}_{H_2\times\{e\}}$, etc.

\subsubsection{Definition of character sheaves}\label{Def of admix}

Following Lusztig, Ginzburg, Grojnowski and Mirkovi\'c-Vilonen \cite{L,G,GR,MV},
we define:
\begin{definition}\label{admis sheaves}

An object $\mM\in D_{K\times K}(G)$ is called a \emph{character sheaf}, if all the irreducible consistuents of its perverse cohomologies appear as irreducible consistuents of perverse cohomologies of objects of the form $CH(\mF) \in D_{K\times K}(G)$, where $\mF\in D_{K\times K}(\mB^2)$. We denote by $D\mathcal{CS}(X/K)\subset D_{K\times K}(G)$ the full subcategory consisting of character sheaves. In other words, $D\mathcal{CS}$ is the full triangulated subcategory of $D_{K\times K}(G)$ generated by objects of the form $CH(\mF) \in D_{K\times K}(G)$ ($\mF\in D_{K\times K}(\mB^2)$) under direct summands (in view of the decomposition theorem). $D\mathcal{CS}(X/K)$ clearly inherits a perverse $t$-structure from $D_{K\times K}(G)$, whose heart is $\mathcal{CS}(X/K) := D\mathcal{CS}(X/K)\cap P_{K\times K}(G)$.

\quash{
A $K$-equivariant perverse sheaf $\mM\in P_{K}(G)$ is called an \emph{admissible sheaf} if
$\mM$ appears as a summand of
$CH(\mF)$ for some $\mF\in P_{K}(\mB^2)$.
We define $D\mA_K(G)$ to be the full subcategory of $D_{K}(G)$ consisting
of those complexes whose perverse constituentst are  admissible sheaves.
The intersection of $D\mA_K(G)$ with $P_K(G)$ will be denoted by $\mA_K(G)$.
}

\quash{
We define the category of admissible sheaves $D\mA(X)\subset D(X)$ on
$X$ to be the image of $D\mA_K(G)$ under the canonical equivalence $D_K(G)\is D(X)$. The intersection of $D\mA(X)$ with
$P(X)$ will be denoted by $\mA(X)$.}
\end{definition}

\quash{
\begin{remark}
The definition of admissible sheaves here is slightly different from that in \cite{L,G,GR,MV}: they
require $\mF\in D_{K\times K}(\mB^2)$.
For example, in the group case, admissible sheaves in \emph{loc. cit.} is the full subcategory of 
$D_{H}(H)$ ($H$ acts on $H$ by the conjugation action) generated by the essential image of the $CH$ functor 
$D_H(\mathcal B_H\times\mathcal B_H)\ra D_H(H)$.
\end{remark}}

\begin{remark}
In the $D$-modules setting, it is proved in \cite{G} that character sheaves are regular holonomic 
$K\times K$-equivariant 
$D$-modules on $G$ such that
 the action of the center $Z$ of the
universal enveloping algebra $U(\fg)$, viewed as invariant differential operators, is locally finite.
\end{remark}

\begin{example}
Consider the group case $(G=H\times H, K=\Delta H)$.
Then we have 
$X=K\backslash G\is H$ and $D_{K\times K}(G)\is D_{H}(H)$, here $H$ acts on 
$H$ by the adjoint action.
The category of character sheaves $D\mathcal{CS}(H)\subset D_{H}(H)$
is the 
full subcategory generated by the essential image of the $CH$ functor (for the group $H$)
\[CH:D_H(\mathcal B_H\times\mathcal B_H)\ra D_{H}(H).\]
\end{example}

Suppose that we are in the setting of \S\ref{symmetric spaces}.
For a $G$-variety $Y$, 
let $D_{MN}(Y)_{um}$ (resp. $D_{M\bar N}(Y)_{um}$)
be the full triangulated subcategory of $D_{MN}(Y)$ (resp. $D_{M\bar N}(Y)$)
generated by the image of $\oblv^P_{MN}:D_{P}(Y) \ra D_{MN}(Y)$ (resp.
$\oblv^{\bar P}_{M\bar N}:D_{\bar P}(Y) \ra D_{M\bar N}(Y)$).
We call $D_{MN}(Y)_{um}$, $D_{M\bar N}(Y)_{um}$
the $A$-(unipotent)-monodromic subcategories.

We have the following properties of character sheaves:

\begin{lemma}\label{A mon}
Let $\mM\in D\mathcal{CS}(X/K)$. We have $\Av_{M\bar N}^K(\mM)\in
D_{M\bar N\times K}(G)_{um}$
and $L_\gamma(\mM)\in D_{MN\times K}(G)_{um}$.
Here $D_{M\bar N\times K}(G)_{um}$ (resp. $D_{M\bar N\times K}(G)_{um}$) is the 
$A$-monodromic subcategories of 
$D_{MN\times K}(G)\is D_{MN}(G/K)$ (resp. $D_{M\bar N\times K}(G)\is D_{M\bar N}(G/K)$).

\end{lemma}

\begin{proof}
It suffices to show that $\Av_{M\bar N}^K(CH(\mF))$ and $L_\gamma(CH(\mF))$
are $A$-monodromic for any $\mF\in P_{K\times K}(\mB^2)$.
Notice that $CH=q_!p^*$ where $q$ is proper and $p$ is smooth. Thus we have $\Av_{M\bar N}^K(CH(\mF))\is CH(\Av_{M\bar N}^K(\mF))$,
and by lemma \ref{LimFunBasicProp} also $L_\gamma(CH(\mF))\is CH(L_\gamma(\mF))$.
Since $CH$ preserves $A$-monodromic subcategories and
$\Av_{M\bar N\times K}^K(\mF)\in D_{M\bar N\times K}(\mB^2)$, $L_\gamma(\mF)\in D_{MN\times K}(\mB^2)$ the
lemma follows from the fact that objects in $D_{MN\times K}(\mB^2)$ and $D_{M\bar N\times K}(\mB^2)$ are $A$-monodromic.
\end{proof}

Recall the following well-known fact:
Consider the adjunction \[\Av_{M\bar N}^{MN}:D_{MN}(G)\rightleftarrows D_{M\bar N}(G):
\av_{MN}^{M\bar N}.\]
Here we regard $G$ as a $G$-variety via the left action.

\begin{lemma}\label{long intertwining}
The pair of adjoint functors ($\Av_{M\bar N}^{MN}$, $\av_{MN}^{M\bar N}$)
defines an equivalence of categories between
$D_{MN}(G)_{um}$ and $D_{M\bar N}(G)_{um}$.
\end{lemma}

\subsection{The formula}\label{Main formula for AD}

\quash{
above
limit functor $L_\gamma$ can be identified with
the limit functor $L_f:D(K\backslash G)\ra D(MN\backslash G)$
for the limit scheme
$f:\widetilde K\backslash\bA^1\times G\ra\bA^1$. }

\quash{

\begin{lemma}\label{mon}
Let $X=G/K$ be the symmetric space.
For any $\mF\in D_{K}(X)$, the object $\Phi_X(\mF)\in D_{MN}(X)$ is $A$-monodromic.
\end{lemma}
\begin{proof}
We will identify the limit functor $\Phi_X$ with the specialization functor in \S\ref{},
and being $A$-monodromic follows from the general properties of specialization functor.
{\red This is not a good proof because we need to assume $G$ is adjoint
in order to consider its compactification.
Can we prove this directly from the definition of $\Phi_X$?
Also is it true for any
$G$-varities $Y$?}
\end{proof}
}

We establish a "transversal fully faithfulness" property for the limit functor in our setting (proposition \ref{Key AM}), and use it to prove the formula for the limit functor, expressing it as a composition of two averaging functors (theorem \ref{formula for L}). The idea of the proof of proposition \ref{Key AM} is that character sheaves, although not living on a proper variety, do "arise", via $CH$, from sheaves on a proper variety, where we do have the "transversal fully faithfulness" property.

\begin{prop}\label{Key AM}

Let $\mM\in D\mathcal{CS}(X/K)$ and $\mE\in D_{M\bar N\times K}(G)_{um}$. Then the natural map \[\Hom_{D_M(G)}(\mM,\mE)\ra\Hom_{D_M(G)}(L_\gamma(\mM),L_\gamma(\mE))\]
is an isomorphism.

\end{prop}

\begin{proof}

Since $D\mathcal{CS}(X/K)$ is generated by perverse constituents of
$CH(\mF)$ ($\mF\in P_{K\times K}(\mB^2)$), and the property of the arrow in the theorem to be an isomorphism, as a property of $\mM$, is stable under retracts and cones, it is enough to show that
\[\Hom_{D_M(G)}(CH(\mF),\mE) \ra \Hom_{D_M(G)}(L_\gamma(CH(\mF)),L_\gamma(\mE))\]
is an isomorphism.

In lemma \ref{kernel} below, we show that the following diagram is commutative
\begin{equation}\label{diagram}
\xymatrix{\Hom_{D_M(G)}(CH(\mF),\mE)\ar[r]^{f\ \ \ \ \ }\ar[ddd]^{f_1}&\Hom_{D_M(G)}(L_\gamma(CH(\mF)),L_\gamma(\mE))\ar[d]^{f_5}\\&\Hom_{D_M(G)}(CH(L_\gamma(\mF)),L_\gamma(\mE))\ar[d]^{f_4}
\\&\Hom_{D_M(\mB^2)}(L_\gamma(\mF),HC(L_\gamma(\mE))
\\\Hom_{D_M(\mB^2)}(\mF,HC(\mE))\ar[r]^{f_2\ \ \ \ }&\Hom_{D_M(\mB^2)}(L_\gamma(\mF),L_\gamma(HC(\mE))\ar[u]_{f_3}
},\end{equation}
where $f_i$ are the canonical morphisms, and $f$ is the map in question.

We now claim that each $f_i$ is an isomorphism (and hence deduce that $f$ is an isomorphism):
1) The maps $f_1$, $f_4$ are adjunction isomorphisms.
2) Since $CH=q_!p^*$ where $q$ is proper and $p$ is smooth, the map $f_5$ is an isomorphism by lemma \ref{LimFunBasicProp}.
3) For $f_3$, we first observe that both $\mE$ and $HC(\mE)$ are $A$-monodromic. Since the
limit functor $L_\gamma$
is isomorphic to the identity functor on the $A$-monodromic subcategory (see lemma \ref{Remark_LimitOfWM}), we have
$\mE\is L_\gamma(\mE)$, $HC(\mE)\is L_\gamma(HC(\mE))$ and under
those isomorphisms the canonical map
$L_\gamma(HC(\mE))\ra HC(L_\gamma(\mE)) $
becomes the identity morphism on $HC(\mE)$. This implies that $f_3$ is an isomorphism.
4) Finally, let us prove that $f_2$ is an isomorphism.
For this, we observe that there are finitely many $K$-orbits (resp. $M\bar N$-orbits) on $\mB$.
Thus, it is clear that $\mF$ (resp. $HC(\mE)$) 
has the form $\mF=\mF_1\boxtimes\mF_2$ (resp. $HC(\mE)=\mG_1\boxtimes\mG_2$), where $\calF_1\in D_{K}(\mB)$ (resp. $\calG_1\in D_{M\bar N}(\mB)$)
Hence, theorem \ref{relative version} (applied to the case $X=Z=\mB$) implies that $f_2$ is an isomorphism (recall that by lemma \ref{Lem:FlagIsMatsuki} and remark \ref{Rem:AdaptedIsLimitGood}, we have the required transversality to apply theorem \ref{relative version}).

\quash{
Proof of 2).
Consider $\mM \in\mA_K(G)$ and $\mE \in D_{\bar U\rtimes\fa}(G)_{um}$. We have:
\beqn
 \Hom_{D_{\bar U}(G)} (Av^{K}_{\bar U} \mM , \mE) \cong\Hom_{D(G)} ( \niceF , \mE) \cong\Hom_{D(G)}(L_{\gamma} \mM, L_{\gamma} \mE)
\cong Hom_{D(G)} (L_{\gamma} \mM, \mE)
\eeqn
\beqn
 \cong Hom_{D_{\bar U}(G)} (Av_{\bar U}^{MN}\circ L_{\gamma} \mM , \mE ).
\eeqn
Here, the first equality is by adjunction and
$Av^{K}_{\bar U} \mM\in D_{\bar U\rtimes\fa}(G)$ (see lemma \ref{A mon}), the second is by part 1), the third is by \ref{LimitFixed} and the fourth is by adjunction and
$Av_{\bar U}^{MN}\circ L_{\gamma}(\mM)\in D_{\bar U\rtimes\fa}(G)$ (see lemma \ref{A mon}).
But this overall equality means, by Yoneda lemma:
$Av^{MN}_{\bar U} \circ L_{\gamma} \cong Av^{K}_{\bar U}$.

}


\end{proof}

In the course of the above proof, we needed the following lemma:

\begin{lemma}\label{kernel}
Diagram (\ref{diagram}) commutes.
\end{lemma}
\begin{proof}
Let $v:CH(\mF)\ra\mE$. Then its image $f_3\circ f_2\circ f_1(v)$ is equal to
\[L_{\gamma}\mF\ra L_{\gamma}\circ HC\circ CH(\mF)\ra L_{\gamma}\circ HC(\mE)\ra HC\circ L_{\gamma}(\mE)\]
where the first arrow is induced by the co-unit map,
the second arrow is induced by $v$,
the third arrow is induced by the natural transformation $L_{\gamma}\circ HC\ra HC\circ L_{\gamma}$.

On the other hand,  $f_4\circ f_5\circ f(v)$ is equal to
\[L_{\gamma}\mF\ra HC\circ CH\circ L_{\gamma}(\mF)\ra HC\circ L_{\gamma}\circ
CH(\mF)\ra HC\circ L_{\gamma}(\mE)\]
where the first arrow is induced by the co-unit map,
the second arrow is induced by the natural transformation $CH\circ L_{\gamma}\ra L_{\gamma}\circ CH$,
the third arrow is induced by $v$.

Notice that, since $L_{\gamma}\circ HC\ra HC\circ L_{\gamma}$ is a natural transformation, we have
the following commutative diagram
\[\xymatrix{L_{\gamma}\circ HC\circ CH(\mF)\ar[r]\ar[d]&L_{\gamma}\circ HC(\mE)\ar[d]
\\ HC\circ L_{\gamma}\circ CH(\mF)\ar[r]&HC\circ L_{\gamma}(\mE)
}.\]
Thus to show that $f_3\circ f_2\circ f_1(v)=f_4\circ f_5\circ f(v)$, it suffices to show
\begin{equation}\label{diagram 1}
\xymatrix{L_{\gamma}(\mF)\ar[r]\ar[d]&L_{\gamma}\circ HC\circ CH(\mF)\ar[d]
\\ HC\circ CH\circ L_{\gamma}(\mF)\ar[r]&HC\circ L_{\gamma}\circ CH(\mF)
}\end{equation}
is commutative.

Recall that $CH=q_!p^*$, $HC=p_*q^!$ and
the natural transformation  $id\ra HC\circ CH=p_*q^!q_!p^*$ factors as
\[id\ra p_*p^*\ra p_*q^!q_!p^*.\]
Using the properties of the limit functor in lemma \ref{LimFunBasicProp}
we have the following diagram
\begin{equation}\label{diagram 2}
\xymatrix{L_{\gamma}(\mF)\ar[r]\ar[d]&L_{\gamma}(p_*p^*(\mF))\ar[d]\ar[r]&L_{\gamma}(p_*q^!q_!p^*(\mF))
\ar[d]\\
p_*p^*L_{\gamma}(\mF)\ar[d]\ar[r]&p_*L_{\gamma}(p^*\mF)\ar[d]\ar[r]&p_*L_{\gamma}(q^!q_!(p^*(\mF))\ar[d]
\\
p_*q^!q_!p^*L_{\gamma}(\mF)\ar[r]&p_*q^!q_!L_{\gamma}(p^*(\mF))\ar[r]&p_*q^!L_{\gamma}(q_!p^*(\mF))}.
\end{equation}
Notice that the outer diagram in (\ref{diagram 2}) is equal to the diagram (\ref{diagram 1}), thus
to show that (\ref{diagram 1}) is commutative it is enough to show that each
of the small diagrams in (\ref{diagram 2}) is commutative. Now the commutativity of
the upper right and lower left diagrams follow from the naturality of the natural transformations
$L_{\gamma} p_*\ra p_*L_{\gamma}$, $id\ra q^!q_!$. The commutativity of the upper left and
lower right diagrams follows from part 3) of lemma \ref{LimFunBasicProp}. This finishes the proof of the lemma.

\quash{
We will prove this by constructing an arrow
$c:\Phi_\mB\circ HC\circ CH(\mF)\ra HC\circ CH\circ\Phi_\mB(\mF)$
making the following diagram commutes
\[\xymatrix{\Phi_\mB(\mF)\ar[r]\ar[d]&\Phi_\mB\circ HC\circ CH(\mF)\ar[d]\ar[ld]_c
\\ HC\circ CH\circ\Phi_\mB(\mF)\ar[r]&HC\circ\Phi_X\circ CH(\mF)
}.\]
For this, let $Z=(G/B_K)\times_X(G/B_K)$ and $p_l,p_r:Z\ra G/B_K$ be the left and right projections.
Let us first construct a natural transformation from
Consider the kernel }
\end{proof}

Finally, here is the main theorem of this paper:
\begin{thm}\label{formula for L}
We have an isomorphism of functors $D\mathcal{CS}(X/K) \ra D_{MN\times K}(G)$:
\[L_{\gamma} \is \av_{MN}^{M\bar N}\circ \Av_{M\bar N}^K.\]

\end{thm}

\begin{corollary}
The functor $\av_{MN}^{M\bar N}\circ \Av_{M\bar N}^K : D\mathcal{CS}(X/K) \to D_{MN\times K}(G)$ is $t$-exact.
\end{corollary}
The same argument as in the poof of corollary \ref{adjoint} gives:
\begin{corollary}
The functor $L_\gamma$ admits both a left and right adjoint.
\end{corollary}

\quash{
\begin{remark}
We have an analogous (unipotent-)monodromic vesion of theorem \ref{formula for L}. See appendix \ref{Monodromic sheaves}.
\end{remark}}
\begin{proof}(of theorem \ref{formula for L})

The proof is the same as of theorem \ref{Thm:FormulaMatsuki};
we first notice that for
$\mM \in D\mathcal{CS}(X/K)$ and
$\mG \in D_{M\bar N\times K}(G)_{um}$
we have:
\beqn
 \Hom_{D_{M\bar N}(G)} (\Av^{K}_{M\bar N} \mM , \mG)\is\Hom_{D_{M}(G)} ( \mM , \niceG) \is\Hom_{D_{M}(G)}(L_{\gamma} \mM, L_{\gamma} \niceG) \cong
\eeqn
\beqn
 \cong \Hom_{D_{M}(G)} (L_{\gamma} \mM, \niceG)  \cong Hom_{D_{M\bar N}(G)} (\Av^{MN}_{M\bar N} L_{\gamma} \niceF , \niceG ).
\eeqn
Here, the first equality is by adjunction
the fact $\Av^{K}_{M\bar N} \mM$ is $A$-monodromic (see lemma \ref{A mon}),
the second is by proposition \ref{Key AM},
the third is by lemma \ref{Remark_LimitOfWM} and the fourth is by adjunction and
the fact $L_{\gamma}(\mM)$, hence $\Av^{MN}_{M\bar N}\circ L_{\gamma}(\mM)$, is $A$-monodromic (see lemma \ref{A mon}). By Yoneda lemma, we get an isomorphism:

\beq\label{third arrow}
\Av^{MN}_{M\bar N} \circ L_{\gamma}\is\Av^{K}_{M\bar N}\text{\quad (as functors $D\mathcal{CS}(X/K) \to D_{M\bar{N}\times K}(G)_{um}$).}
\eeq

Now we compose both sides of (\ref{third arrow}) with
$av^{M\bar N}_{MN}$:

$$\av^{M\bar N}_{MN} \circ\Av^{MN}_{M\bar N} \circ L_{\gamma} \cong\av^{M\bar N}_{MN} \circ \Av^K_{M\bar N} \text{\quad (as functors $D\mathcal{CS}(X/K) \to D_{MN\times K}(G)_{um}$).}$$

Notice that on the left hand side, we apply $\av^{M\bar N}_{MN} \circ\Av^{M\bar N}_{MN}$ to objects in $D_{MN\times K}(G)_{um}$. Hence, by lemma \ref{long intertwining}, this application has no effect, and we get:

$$ L_{\gamma} \cong\av^{M\bar N}_{MN} \circ\Av^K_{M\bar N} \text{\quad (as functors $D\mathcal{CS}(X/K) \to D_{MN\times K}(G)_{um}$).}$$

The proof is completed.

\end{proof}

\appendix

\section{Quasi-Affineness of $\widetilde{G} / \widetilde{K}$}\label{App:QuasiAff}

In this appendix, we follow very closely \cite[appendix C]{DG}.

\subsection{Statement}

\subsubsection{}

Fix a connected reductive group $G$ over an algebraically closed field $k$, and an involution $\theta:G \ra G$. Write as usual $G^{\theta}=\{g\in G \ | \ \theta(g)=g\}$. Fix an open subgroup $K$ of $G^{\theta}$.

A torus $S \subset G$ is called $\theta$-\emph{split} if $\theta(s)=s^{-1}$ for all $s\in S$.
A parabolic subgroup $P \subset G$ is called $\theta$-\emph{split} if $L := P \cap \theta(P)$ is a Levi subgroup of $P$ (and of $\theta(P)$). Taking $A$ to be the maximal $\theta$-split torus in $Z(L)$, one has $L = Z_G (A)$.

Fix a minimal $\theta$-split parabolic $P \subset G$, and thus the corresponding $L := P \cap \theta(P)$ and $A$ - the maximal $\theta$-split torus in $Z(L)$. Note that $A$ is a maximal $\theta$-split torus in $G$, since $P$ is minimal.

Denote $M := L \cap K = Z_K(A)$, and denote $N := R_u (P)$. One has the decomposition $L = MA$, and thus the Langlands decomposition $P=MAN$.

Fix also a co-character $\gamma : \bG_m \to A$, which we suppose to have negative pairing with roots of $N$.

Additionally, choose a maximal torus $T$ of $L$, and a Borel subgroup $T \subset B \subset P$ (so that we can talk about highest weights etc.). One has $A \subset T$.

\subsubsection{}

Consider the constant group scheme $\widetilde{G} := \bA^1 \times G \to \bA^1$. Consider the following closed subscheme of $G^{\circ} = \bG_m \times G$: $$K^{\circ} = \{ (t,g) \ | \ g \in \gamma(t) K \gamma(t)^{-1} \},$$ and let $\widetilde{K}$ be the closure of $K^{\circ}$ in $\widetilde{G}$.

By \cite[proposition 2.3.8]{DG} or \cite{AM1}, $\widetilde{K} \ra \bA^1$ is a smooth group scheme over $\bA^1$. Denoting by $K_t$ its fiber over $t \in \bA^1(k)$, one has $K_t = \gamma(t) K \gamma(t)^{-1}$ for $t \neq 0$, and $K_0 = MN$.

The following result is the goal of this appendix.

\begin{claim}\label{cl:appc}

The quotient space $ \widetilde{K} \backslash \widetilde{G} \to \bA^1$ is a quasi-affine scheme.

\end{claim}

\subsection{Reduction of claim \ref{cl:appc} to claim \ref{ClaimExistsGSpace}}

First of all, as $\widetilde{K} \backslash \widetilde{G} \to \widetilde{G^{\theta}} \backslash \widetilde{G}$ is finite, we can (and will) assume that $K = G^{\theta}$.

As is explained in \cite[$C.1.1$ and $C.1.2$]{DG}, claim \ref{cl:appc} follows from the following claim:

\begin{claim}

$\widetilde{K} \backslash \widetilde{G} \to \bA^1$ admits a quasi-finite map to an affine scheme $Z \to \bA^1$.

\end{claim}

This last claim is unfolded to the following claim:

\begin{claim}\label{ClaimExistsGSpace}

There exists an affine scheme $Y$, with an right action of $G$, and a morphism $f : \bA^1 \to Y$, such that for every $t \in \bA^1(k)$, the stabilizer $Stab_{G(k)} ( f(t) )$ contains $K_t (k)$ as a subgroup of finite index.

\end{claim}

\subsection{Proof of claim \ref{ClaimExistsGSpace}}

To construct $Y$, we will consider finite-dimensional representations $V$ of $G$ (write $\rho : G \to Aut(V) \subset End(V)$ for the corresponding homomorphism). We then set $Y := End(V)$, and the right $G$-action on $Y$ we set to be $ T \cdot g :=  \rho(\theta(g)^{-1}) \circ T \circ \rho(g)$.  We also choose an integer $n$, and consider the morphism $f : \bG_m \to Y$ defined by $f(t) = t^{2n} \rho (\gamma(t^{-2}))$. If $n$ is big enough, it extends to a morphism $f: \bA^1 \to Y$. More precisely, $n$ should be not less than $ \langle \gamma , \omega \rangle$, for all $A$-weights $\omega$ of $V$ (and we assume that $n$ satisfies this in what follows).

Let us first find the general shape of the desired stabilizers for this construction, and check that $K_t (k) \subset Stab_{G(k)}(f(t))$.

For $\bA^1(k) \ni t \neq 0$, $Stab_{G(k)} (f(t))$ is the subgroup of $G(k)$ contsisting of elements $g$ satisfying

\begin{equation}
\label{eq:neq0} \rho(\theta(g)) = \rho(\gamma(t^{-2}) g \gamma(t^{2})) .
\end{equation}

Since $K_t (k) = \{ g \in G(k) \ | \ \theta(g) = \gamma(t^{-2}) g \gamma(t^{2}) \}$, it is clear that $K_t (k) \subset Stab_{G(k)} (f(t))$.

For $t = 0$, notice that $f(0)$ is a projection operator, with kernel $V^{\gamma < n}$ and image $V^{\gamma \ge n} = V^{\gamma = n}$ (where $V^{\gamma < n}$, for example, stands for the sum of all $A$-weight subspaces with weight $\omega$ satisfying $\langle \gamma , \omega \rangle < n$). Then $Stab_{G(k)} (f(0))$ is seen to be the subgroup of $G(k)$ consisting of elements $g$ satisfying the following three properties:

\begin{equation}\label{eq:eq0}
\rho(g) V^{\gamma < n} \subset V^{\gamma < n}, \rho(\theta(g)) V^{\gamma \ge n} \subset V^{\gamma \ge n}  , ( \rho(g) - \rho(\theta(g)) ) V^{\gamma \ge n} \subset V^{\gamma < n}.
\end{equation}

From this, it is clear that $K_0(k) = N(k) M(k) \subset Stab_{G(k)} (f(0))$.

\subsubsection{}

Taking the product of the various $G$-varieties $Y$ and morphisms $f:\bA^1 \to Y$ that we will study in paragraphs \ref{subsubsec:c1}, \ref{subsubsec:c2} and \ref{subsubsec:c3}, we get a $G$-variety $Y$ and morphism $f:\bA^1 \to Y$ which satisfy the wanted condition on stabilizers.

\subsubsection{}\label{subsubsec:c1}

Assume that $V$ is a faithful $G$-representation, and take any big enough $n$.

Then equation \ref{eq:neq0} becomes just $\theta(g) = \gamma(t^{-2}) g \gamma(t^{2})$, so it is clear that $Stab_{G(k)} (f(t)) = K_t (k)$, whenever $t \neq 0$.

\subsubsection{}\label{subsubsec:c2}

Assume that $V$ is an irreducible $G$-representation with lowest weight $\lambda$, such that $\langle \alpha^{\vee} , \lambda \rangle \neq 0 $ for all co-roots $\alpha^{\vee}$. Set $n = \langle \gamma , \lambda \rangle$.

By the lemma below, the condition $\rho(\theta(g)) V^{\gamma \ge n} \subset V^{\gamma \ge n}$ which appears in \ref{eq:eq0} implies $\theta(g) \in \theta(P)(k)$. Thus, we get $g \in P(k)$. So, for this $Y$, $Stab_{G(k)} (f(0))$ is contained in $P(k)$.

\begin{lemma}

The stabilizer of $V^{\gamma \ge n}$ in $G(k)$ is $\theta(P)(k)$.

\end{lemma}

\begin{proof}

Denote by $Q$ the stabilizer of $V^{\gamma \ge n}$ in $G(k)$. Notice that $\theta(P)(k) \subset Q$, and we would like to show that $\theta(P)(k) = Q$. As $Q$ contains a parabolic, it is a parabolic itself. Suppose that $\theta(P)(k) \neq Q$. We would then find a root $\alpha$ of $\theta(N)$ (i.e. a root of $\theta(P)$ and not of $M$) such that $\alpha$ and $1 / \alpha$ are roots of $Q$. Denote by $u_{\alpha} , u_{1 / \alpha} : \bG_a \to Q$ the corresponding root subgroups. Denote by $v$ a the lowest weight vector in $V$. Since $\gamma$ satisfies $\langle \gamma , \alpha \rangle > 0$, $v - u_{1 / \alpha}(x)v$ lies in $V^{\gamma < n}$, which means, by the definition of $Q$, that $v - u_{1 / \alpha}(x)v = 0$, i.e. $u_{1 / \alpha}(x)$ fixes $v$. But then, since $u_{\alpha}(x)$ also fixes $v$ (since $v$ is a lowest weight vector), considering the relevant $SL_2$-triple, we see that the co-root subgroup $\alpha^{\vee} : \bG_m \to T$ fixes $v$, i.e. $\langle \alpha^{\vee} , \lambda \rangle = 0$. But this contradicts our assumption on $\lambda$.

\end{proof}

\subsubsection{}\label{subsubsec:c3}

Let $\lambda : T \to \bG_m$ be an anti-dominant root, and $V_{\lambda}$ an irreducible representation of $G$ with lowest weight $\lambda$. Set $n = \langle \gamma , \lambda \rangle$ and consider the corresponding $Y_{\lambda} := End(V_{\lambda})$. Recalling that $\theta(a) = a^{-1}$ for $a \in A(k)$ and explicating the last condition of \ref{eq:eq0} we see that $Stab_{G(k)} (f(0)) \cap A(k)$ is contained in $\{ a \in A(k) \ | \ \lambda(a)^2=1 \}$. By the next lemma, taking $Y := Y_{\lambda_1} \times \cdots \times Y_{\lambda_r}$, we will have that $Stab_{G(k)} (f(0)) \cap A(k)$ is finite.

\begin{lemma}

There exist finitely many anti-dominant $\lambda_1 , \ldots , \lambda_k \in X^*T$ such that their joint kernel in $T$ is finite.

\end{lemma}


\begin{thebibliography}{99}

\bibitem[AM1]{AM1} N. Abe, Y. Mieda: {\it Remarks on geometric Jacquet functors},
J. Math. Sci. Univ. Tokyo, 17 (2010), 243-246.


\bibitem[AM2]{AM2} N. Abe, Y. Mieda: {\it Jacquet functor and De Concini-Procesi compatification},
Int Math Res Notices (2015) Vol. 2015 3810-3829.


\bibitem[An]{An} S. Anantharaman: {\it Sch\'emas en groupes, espaces homog\`enes et espaces alg\'ebriques sur une base de dimension 1},
Bull. Soc. Math. France, M\'emoire 33 (1973), 5–79.

\quash{
\bibitem[AB]{AB} S. Arkhipov, R. Bezrukavnikov: {\it Perverse sheaves on affine flags and Langlands dual group},
Israel Journal of Mathematics, 170 (2009), 135-183.}

\quash{\bibitem[Ay]{Ay} J. Ayoub: {\it The motivic nearby cycles and the conservation conjecture},
Algebraic Cycles and Motives, Vol 1, 3-54, London Math. Soc. Lecture Note Ser., 343 (2007), Cambridge Univ. Press.}

\bibitem[B]{B} T. Braden: {\it Hyperbolic localization of Intersection Cohomology}, Transformation Groups 8 (2003), no. 3, 209–216. Also arXiv:math/020225


\bibitem[BB1]{BB1} A. Beilinson, J. Bernstein: {\it A generalization of Casselman's submodule theorem},
Representation theory of reductive groups, 35-52, Progr. Math. 40, Birkh$\ddot{\on{a}}$user Boston, MA, 1983.

\bibitem[BB2]{BB2} A. Beilinson, J. Bernstein: {\it A proof of Jantzen Conjectures},
 I. M. Gelfand Seminar, 1�1�750, Adv. Soviet Math., 16, Part 1, Amer. Math. Soc., Providence, RI, 1993.

\bibitem[BFO]{BFO} R. Bezrukavnikov, M. Finkelberg, V. Ostrik: {\it Character D-modules via Drinfeld center
of Harish-Chandra bimodles},
Invent. Math. 188 (2012), 589-620.

\bibitem[BL]{BL} J. Bernstein, V. Lunts: {\it Equivariant sheaves and functors},
Lecture Notes in Math. 1578 (1994).

\bibitem[BK]{BK} R. Bezrukavnikov, D. Kazhdan: {\it Geometry of second adjointness for $p$-adic group},
arXiv:1112.6340v2.



\quash{
\bibitem[CG]{CG} N. Chriss, V. Ginzburg: {\it Representation theory and complex geometry},
Birkh$\ddot{\on{a}}$user Boston, Cambridge, MA, 1997.}

\bibitem[DP]{DP} C. De Concini, C. Procesi.: {\it Complete symmetric varieties},
Lecture Notes in Mathematics Volume 996, 1983, 1-44.

\quash{\bibitem[D]{D} A. Dimca: {\it Sheaves in Topology},
Universitext, Springe-Verlag, Berlin, 2004.}

\bibitem[DG1]{DG} V. Drinfeld, D. Gaitsgory: {\it Geometric constant term functor(s)},
arXiv:1311.2071.

\bibitem[DG2]{DG2} V. Drinfeld, D. Gaitsgory: {\it On a theorem of Braden},
arXiv:1308.3786.

\bibitem[ENV]{ENV} M. Emerton, D. Nadler, K. Vilonen: {\it A geometric Jacquet functor},
Duke Math. J. 125 (2004), no. 2, 267-278.

\bibitem[G]{G} V. Ginzburg: {\it Admissible modules on a symmetric space},
 Ast\'erisque, 173-174 (1989), 199-255.

\quash{ \bibitem[GD]{GD} D. Gaitsgory, D. Nadler: {\it
Spherical varieties and Langlands duality
},
Moscow Mathematical Journal,
Volume 10, Number 1, 65-137 (2010). }

\bibitem[GR]{GR} I. Grojnowski: {\it Character sheaves on symmetric spaces},
Phd Thesis, available at https://www.dpmms.cam.ac.uk/~groj/papers.html

\quash{
\bibitem[HTT]{HTT} R. Hotta, K. Takeuchi, T. Tanisaki: {\it $D$-modules, perverse sheaves, and representation theory},
Progress in Mathematics, Vol 236.}

\quash{\bibitem[K]{K} V. YU. Kaloshin: {\it A geometric proof of the existence of Whitney stratifications
}, Moscow Mathematical Journal,
Volume 5, Number 1, Pages 125-133.}

\bibitem[L]{L} G. Lusztig: {\it Character sheaves I-V},
Adv. Math. 56 (1985), 193-237; Adv. Math. 57 (1985), 226-265;
Adv. Math. 57 (1985), 266-315; Adv. Math. 59 (1986), 1-63;
Adv. Math. 61 (1986), 103-115.
\quash{
\bibitem[L2]{L2} G. Lusztig: {\it Character sheaves II},
Adv. Math. 57 (1985), 226-265.

\bibitem[L3]{L3} G. Lusztig: {\it Character sheaves III},
Adv. Math. 57 (1985), 266-315.

\bibitem[L4]{L4} G. Lusztig: {\it Character sheaves IV},
Adv. Math. 59 (1986), 1-63.

\bibitem[L5]{L5} G. Lusztig: {\it Character sheaves V},
Adv. Math. 61 (1986), 103-115.
}

\bibitem[MV]{MV} I. Mirkovi\'c, K. Vilonen: {\it Characteristic varieties of character sheaves},
Invent. math. 93, 405-418 (1988).

\bibitem[MUV]{MUV} I. Mirkovi\'c, T. Uzawa, K. Vilonen: {\it Matsuki correspondence for sheaves},
Invent. math. 109, 231-245 (1992).

\bibitem[N]{N} D. Nadler: {\it  Morse theory and tilting sheaves},
Pure App. Math. Q. 2 (2006), no. 3, (Special Issue: In honor of Robert MacPherson, Part 1 of 3), 83--108.

\bibitem[SV]{SV} Y. Sakellaridis, A. Venkatesh: {\it Periods and harmonic analysis on spherical varieties},
arXiv:1203.0039.

\quash{\bibitem[S]{S} L. Sch$\ddot{\on{u}}$rmann: {\it Topology of singular spaces and constructible sheaves},
Monografie Matematyczne Vol. 63.}


\quash{\bibitem[Vust]{Vust} T. Vust: {\it Op{\'e}ration de groupes r\'eductifs dans un type de c$\hat{o}$nes presque homog\'enes},
Bull. Soc. Math. France 102, 317-334 (1974).}

\quash{\bibitem[V]{V} E. B. Vinberg: {\it On reductive algebraic semigroups},
E. B. Dynkin's Seminar, 145-182, Amer. Math. Soc. Transl. Ser. 2, 169,
Amer. Math. Soc., Providence, RI, 1995.}

\end{thebibliography}
\end{document}